\documentclass[12pt]{amsart}
\usepackage{amssymb,amsmath,amsfonts,latexsym}
\usepackage{bm}
\usepackage[all,cmtip]{xy}
\usepackage{amscd}

\newcommand{\newsection}[1]{\setcounter{equation}{0} \section{#1}}
\newcommand{\bea}{\begin{eqnarray}}
\newcommand{\eea}{\end{eqnarray}}
\newcommand{\dsp}{\displaystyle}

\newcommand{\clb}{\mathcal{B}}

\newcommand{\cle}{\mathcal{E}}

\newcommand{\clh}{\mathcal{H}}
\newcommand{\clk}{\mathcal{K}}

\newcommand{\clm}{\mathcal{M}}
\newcommand{\cln}{\mathcal{N}}

\newcommand{\clr}{\mathcal{R}}
\newcommand{\cls}{\mathcal{S}}

\newcommand{\clu}{\mathcal{U}}
\newcommand{\clw}{\mathcal{W}}

\newcommand{\z}{\bm{z}}
\newcommand{\w}{\bm{w}}

\newcommand{\D}{\mathbb{D}}

\newcommand{\raro}{\rightarrow}

\newcommand{\mo}{\mathop{\oplus}}

\def\textmatrix#1&#2\\#3&#4\\{\bigl({#1 \atop #3}\ {#2 \atop #4}\bigr)}
\def\dispmatrix#1&#2\\#3&#4\\{\left({#1 \atop #3}\ {#2 \atop #4}\right)}
\newcommand{\be}{\begin{equation}}
\newcommand{\ee}{\end{equation}}
\newcommand{\ben}{\begin{eqnarray*}}
	\newcommand{\een}{\end{eqnarray*}}

\newcommand{\NI}{\noindent}

\newcommand{\bi}{\begin{itemize}}
	\newcommand{\ei}{\end{itemize}}

\theoremstyle{definition}

\newtheorem*{theorem*}{Theorem}

\theoremstyle{plain}

\newtheorem{thm}{Theorem}[section]
\newtheorem{cor}[thm]{Corollary}
\newtheorem{lem}[thm]{Lemma}
\newtheorem{prop}[thm]{Proposition}
\theoremstyle{definition}
\newtheorem{defn}[thm]{Definition}
\newtheorem{rem}[thm]{Remark}
\newtheorem{ex}[thm]{Example}

\numberwithin{equation}{section}

\let\phi=\varphi

\begin{document}

\title[Wold decomposition for isometries with equal range]	
{Wold decomposition for isometries with equal range}

\author[Majee]{Satyabrata Majee}
\address{Indian Institute of Technology Roorkee, Department of Mathematics,
		Roorkee-247 667, Uttarakhand,  India}
\email{smajee@ma.iitr.ac.in}

\author[Maji]{Amit Maji}
\address{Indian Institute of Technology Roorkee, Department of Mathematics,
		Roorkee-247 667, Uttarakhand,  India}
\email{amit.maji@ma.iitr.ac.in, amit.iitm07@gmail.com}




\begin{abstract}
Let $n \geq 2$, and let $V=(V_1,\dots, V_n)$ be an $n$-tuple of isometries acting on a Hilbert space $\clh$. We say that $V$ is an $n$-tuple of {\textit{isometries with equal range}} if $V_i^{m_i}V_j^{m_j}\clh = V_j^{m_j} V_i^{m_i}\clh $ and $V_i^{*m_i}V_j^{m_j} \clh= V_j^{m_j} V_i^{*m_i}\clh$ for $m_i,m_j \in \mathbb{Z}_+$, where $1 \leq i<j \leq n$.

\NI
We prove that each $n$-tuple of {\textit{isometries with equal range}} admits a unique {\textit{Wold decomposition}}. We obtain analytic models of the above class, and as a consequence, we show that the wandering data are complete unitary invariants for $n$-tuples of {\textit{isometries with equal range}}. Our results unify all prior findings on the decomposition for tuples of isometries in the existing literature.
\end{abstract}	
\maketitle

\newsection{Introduction}

One of the key problems in operator theory, operator algebras, and analytic function theory is the classification and representation of $n$-tuples ($n \geq 2$) of commuting isometries acting on Hilbert spaces. The foundation of this study is the complete and explicit structure theory of a single isometry. It says that each isometry is the direct sum of a unitary operator and copies of the unilateral shift, and we shall refer to this classical decomposition as the {\textit{Wold decomposition or Wold-von Neumann decomposition}} (see details in Section 2). The aim of this paper is to obtain {\it{Wold decomposition}} for a large class of $n$-tuples of isometries (not necessarily commuting or twisted) on Hilbert spaces.

{\textit{Wold decomposition}} plays a crucial role in many areas of operator algebras and operator theory, namely, dilation theory, invariant subspace theory, operator model \cite{NF-BOOK}; and also prediction theory \cite{{HL-Prediction-theory}}, time series analysis \cite{MW-prediction-theory}, stochastic process \cite{MP-Wold-decomposition} etc. Thus, it is a natural query to ask whether Wold's model can be extended to several variables. Nevertheless, this generalization is not simple, and for example, in the case of commuting isometries, the structure is quite complex and mainly unknown. There has been a lot of research in this area over the past few decades, and a variety of significant and intriguing findings have been made. The structure of an $n$-tuple ($n \geq 2$) of commuting isometries on a Hilbert space is still not completely known. Many researchers investigated Wold-type decomposition for a pair or family of commuting isometries, like, Suciu \cite{SUCIU-SEMIGROUPS} developed a structure theory for a semigroup of isometries. Later, S\l{}o\'{c}inski \cite{SLOCINSKI-WOLD} obtained a Wold-type decomposition for pairs of doubly commuting isometries from Suciu's decomposition of the semigroup of isometries and Popovici \cite{POPOVICI-WOLD TYPE} studied the Wold-type decomposition for pairs of commuting isometries. Recently, Jeu and Pinto \cite{JP-NON COMMUTING} established {\textit{Wold decomposition}} for $n$-tuple ($n \geq 2$) of doubly non-commuting isometries, that is, an $n$-tuple of isometries $(V_1, \dots, V_n)$ on $\clh$ satisfying the following conditions:
For all  $i,j$ with $1 \leq i,j \leq n$, $z_{ij} \in S^1:=\{z_{ij} \in \mathbb{C}: 
|z_{ij}|=1\}$ are given such that $z_{ji}=\overline{z}_{ij}$ for all $i \neq j$, 
and 
\[
V_iV_j={z}_{ij} V_j V_i  \hspace{1cm} \text{and} \hspace{1cm} V_i^*V_j=
 \overline{z}_{ij} V_j V_i^*
\]
for all $i \neq j$. An $n$-tuple of isometries $(V_1, \dots, V_n)$ on $\clh$
becomes doubly commuting if $z_{ij}=1$ for all $i, j$ and a $\clu_n$-twisted isometry if $z_{ij}$'s are replaced by unitaries $U_{ij}$, where each $V_k$ is in the commutant of $U_{ij}$ for $i \neq j $ and $k =1, \ldots, n$. 
Inspiring from the {\textit{Wold-von Neumann decomposition}} for a single isometry, Rakshit, Sarkar, and Suryawanshi \cite{RSS-TWISTED ISOMETRIES} formulated the following definition for $n$-tuples of isometries:

\begin{defn}\label{Orthogonal-def}
Let $V=(V_1,  \dots, V_n)$ be any $n$-tuple of isometries on a Hilbert space
$\clh$. Then we call that the $n$-tuple $V$ admits {\textit{Wold-von Neumann decomposition}} (or {\it{Wold decomposition}} in short)
if there exist $2^n$-closed subspaces $\{\clh_{\Lambda}\}_{\Lambda \subseteq I_n}$ of 
$\clh$ (including trivial subspaces) such that 
\begin{itemize}
\item [(a)] $\clh_{\Lambda}$ reduces each isometry $V_i$ for $i = 
1,2, \dots, n$ and $\Lambda \subseteq I_n = \{1,2,\dots, n\}$,
\item [(b)] $\clh= \bigoplus\limits_{\Lambda \subseteq I_n } 
\clh_{\Lambda}$,
\item [(c)] for each $\Lambda \subseteq I_n$, $V_i|_{\clh_{\Lambda}}$ 
is a shift for $i \in \Lambda$, and $V_j|_{\clh_{\Lambda}}$ is unitary for 
$j \in I_n \setminus \Lambda$.
\end{itemize}  	 
\end{defn}
\NI
The authors showed that each $\clu_n$-twisted isometry admits {\it{Wold decomposition}} in the above sense, and also obtained analytic models of $\clu_n$-twisted isometries. 
\vspace{0.1cm}

Let $n \geq 2$ and $V=(V_1,\dots, V_n)$ be an $n$-tuple of isometries on $\clh$ such that 
$V_i^{m_i}V_j^{m_j}\clh = V_j^{m_j} V_i^{m_i}\clh $ and $V_i^{*m_i}V_j^{m_j} \clh= V_j^{m_j} V_i^{*m_i}\clh$ for $m_i,m_j \in \mathbb{Z}_+$, where $1 \leq i<j \leq n$.
We shall refer $V=(V_1,\dots, V_n)$ as an $n$-tuple of {\textit{isometries with equal range}}. Clearly, $n$-tuples $(n \geq 2)$ of doubly non-commuting isometries and $\clu_n$-twisted isometries are $n$-tuples of {\it{isometries with equal range}}. One of the central results of this paper is that each $n$-tuple of {\it{isometries with equal range}} admits {\textit{Wold-von Neumann decomposition}}. Indeed, we have the following result (see Section 5 below):

\begin{theorem*}\label{Theorem Main Result}
Let $V=(V_1,\dots, V_n)$ be an $n$-tuple ($n \geq 2$) of isometries on $\clh$ such that $V_i^{m_i}V_j^{m_j}\clh = V_j^{m_j} V_i^{m_i}\clh $ and $V_i^{*m_i}V_j^{m_j} \clh= V_j^{m_j} V_i^{*m_i}\clh$ for  $m_i,m_j \in \mathbb{Z}_+$, where $1 \leq i<j \leq n$. Then there exist $2^n$ joint $V$-reducing subspaces $\{\clh_{\Lambda}: \Lambda \subseteq I_n \}$ (counting the trivial subspace $\{0\}$) such that 
\begin{align*}
\clh= \bigoplus_{\Lambda \subseteq I_n} \clh_{\Lambda},
\end{align*}
where, for  $\clh_{\Lambda} \neq \{0\}$, $V_j|_{\clh_{\Lambda}}$ is unitary if 
$j \in I_n \backslash \Lambda$ and $V_i|_{\clh_{\Lambda}}$ is a shift if $i \in \Lambda$. Moreover, for each 
$ \Lambda \subseteq I_n $, 
\begin{align}\label{equn 4}
\clh_{\Lambda}= \bigoplus_{{\bm p} \in \mathbb{Z}_{+}^{|\Lambda|}} V_{{ \Lambda}}^{\bm p} 
\bigg( \bigcap_{{\bm q} \in \mathbb{Z}_{+}^{{n - |\Lambda|}}}V^{\bm q}_{ I_n \backslash \Lambda} \ \clw_{ \Lambda}  \bigg),
\end{align}	
and the above decomposition is unique.
\end{theorem*}

It is important to note that our results, restricted to the tuples of doubly commuting or $\clu_n$-twisted isometries, recover the classification results. Also, orthogonal decomposition for the non-commuting isometries will be fruitful as applications to the $C^{*}$-algebra context, see \cite{JPS-q COMMUTING ISOMETRIES}, \cite{Weber-algebras}.

The outline of this article is organized as follows: In Section 2, we discuss some basic definitions and the classical {\it{Wold-von Neumann decomposition}} for a single isometry. We also supplement some fundamental results that, although simple, are essential to determining the {\it{Wold decomposition}} for tuples of isometries. 
In Section 3, we present some important auxiliary results and the general structure of $n$-tuples of isometries. Section 4 is concerned with the {\it{Wold decomposition}} for pairs of {\textit{isometries with equal range}}. Section 5 is devoted to the {\it{Wold decomposition}} for $n$-tuple of {\textit{isometries with equal range}}. A concrete analytic model for an $n$-tuple of {\textit{isometries with equal range}} has been presented in Section 6. Finally, in Section 7 we obtain that the wandering data are complete unitary invariants for $n$-tuples of {\textit{isometries with equal range}}. Our findings are more broadly applicable as they cover a large class of tuples of isometries, and on the other hand, our method is new and notably distinct from earlier research.

\section{Notations and Preparatory results}

In this section, we recall some basic definitions and some fundamental results that will be used frequently in the sequel.

Throughout this article, $\mathbb{Z}_{+}$ denotes the set of all non-negative integers, $\clh$ denotes a separable Hilbert space over complex field, $\mathcal{B}(\clh)$ is the $C^*$-algebra of all bounded linear operators (operators in short) on $\clh$, and $I_{\clh}$ stands for an identity operator on $\clh$. A subspace $\mathcal{M}$ of $\mathcal{H}$ is invariant under 
$V \in \clb(\clh)$ if $V(\mathcal{M}) \subseteq \mathcal{M}$ and subspace 
$\mathcal{M}$ reduces $V$ if $V(\mathcal{M}) \subseteq \mathcal{M}$ and 
$ V(\mathcal{M}^{\perp}) \subseteq \mathcal{M}^{\perp}$. We denote $\cln(V)$ 
and $\clr(V)$ as the kernel of $V$ and range of $V$, respectively. We shall frequently use the basic properties of an isometry $V$ (that is, $V^{*}V =I_{\clh}$) on $\clh$ namely, $V^*\clh = \clh$ and $\clr(V) = \cln(V^*)^{\perp}$. A \textit{wandering subspace} of an isometry $V$ is a closed subspace $\mathcal{W} \subseteq \mathcal{H}$ such that 
\[
V^{k} \mathcal{W} \perp V^{l} \mathcal{W}  ~~ \ \text{for all}
 \ \  {k , l} \ \in \   \mathbb{Z}_{+} \ \text{with} \ k \neq l. 
\]
An isometry $V$ on $\mathcal{H}$ is said to be a {\textit {unilateral shift 
or shift}} if $ \mathcal{H}= \bigoplus_{m \geq 0} V^m \mathcal{W}$ for 
some wandering subspace $\mathcal{W}$ of $V$; equivalently, an isometry 
$V$ is called a shift if $V^{*m} \rightarrow 0$ as $m \rightarrow \infty$ 
in the strong operator topology (see Halmos \cite{HALMOS-BOOK}). If $\clw$ 
is a \textit{wandering subspace} of a shift $V$ on $\clh$, then
\[
\clw = \clh \ominus V \clh
\]
and the multiplicity of the shift $V$ is defined as the dimension of $\clw$
(see \cite{NF-BOOK}).
\vspace{.2cm}

One of the fundamental results in classification theory is the \textit{Wold-von Neumann decomposition} of an isometry on a Hilbert space due to H. Wold \cite{WOLD-TIME SERIES} and J. von-Neumann \cite{von Neumann}.

\begin{thm}\label{thm-Wold}
Let $V$ be an isometry acting on a Hilbert space $\clh$ and $\clw = \clh \ominus V \clh$.
Then $\clh$ decomposes uniquely as a direct sum of two $V$-reducing subspaces
$\clh_s = \displaystyle{\mo_{m=0}^\infty} V^m \clw$ and
$\clh_u = \clh \ominus \clh_s = \dsp \cap_{m=0}^{\infty} V^m \clh$ and
\[
V = \begin{bmatrix} V_s & O\\ O & V_u
\end{bmatrix} \in \clb(\clh_s \oplus \clh_u),
\]
where $V_s = V|_{\clh_s}$ is a shift and $V_u = V|_{\clh_u}$ is unitary.
\end{thm}

Before moving on to the main context, we need the following simple and fundamental results:

\begin{lem}\label{Equal-Range-Isometry}
Let $(V,W)$ be a pair of isometries on $\clh$ such that $\clr(V)=\clr(W)$. Then $VV^*=WW^*$. 
\end{lem}

\begin{proof}
Suppose that $(V,W)$ is a pair of isometries on $\clh$ such that 
$\clr(V)=\clr(W)$. Now $ [\clr (V)]^\perp=[\clr(W)]^\perp$ implies 
$\cln(V^*)= \cln(W^*)$. Also $ \clr(I_{\clh}-VV^*)= \cln (V^*)$ gives 
$I_{\clh}-VV^*=P_{\cln(V^*)}$. Therefore,
$VV^*= I_{\clh} -P_{\cln (V^*)}= I_{\clh} - P_{\cln (W^*)}= WW^*.$
\end{proof}

\begin{lem}\label{Lemma-Unitary-equal range}
Let $(V,W )$ be a pair of isometries on $\clh$ such that $\clr(V)=\clr(W)$. Then 
there exists a unique unitary operator $U \in \clb(\clh)$ such that $V=WU$.
\end{lem}

\begin{proof}
Suppose $(V,W )$ is a pair of isometries on $\clh$ such that  $V\clh = W\clh$. 
Then for each $x \in \clh$, there exists a unique $y \in \clh$ such that $Vx=Wy$. 
Now define a map $U:\clh \rightarrow \clh$ as $Ux= y$. Clearly, $U$ is well-defined and linear on $\clh$. Thus 
\[
Vx = WUx \quad ( x \in \clh ).
\] 
Since $V$ and $W$ are isometries on $\clh$, for each $x \in \clh$ 
\[
\|U x\| = \| WUx \|  = \|Vx\| =\|x\|.
\] 
Therefore, $U$ is an isometry and it is also easy to see that $U$ is onto. Thus $U$ 
is a unitary operator on $\clh$ such that $Vx = WUx$ for all $x \in \clh$. Hence $V=WU$.

Suppose that there exist two unitaries $U_1 , U_2 \in \clb(\clh)$ such that $V=WU_1$ and $V=WU_2$. Then $U_1 = W^*V = U_2$. This completes the proof. 
\end{proof}

\begin{lem}\label{Reducing-subspace-with-equal-range}
Let $(V, W)$ be a pair of isometries acting on $\clh$ such that $\clr(V^m W^n) =\clr(W^nV^m)$ for all $m,n \in \mathbb{Z}_+$. If $\clm \subseteq \clh$ reduces both the isometries $V$ and $W$, then $V^m W^n\clm =W^nV^m \clm$ for all $m,n \in \mathbb{Z}_+$.
\end{lem}

\begin{proof}
Assume that $\clm \subseteq \clh$ reduces both the isometries $V$ and $W$. Then 
$V|_{\clm}$ and $W|_{\clm}$ are isometries on $\clm$. Since $\clr(V^m W^n)=\clr(W^nV^m)$, using Lemma \ref{Equal-Range-Isometry} we obtain 
\[
V^m W^n W^{*n}V^{*m} =W^n V^m V^{*m}W^{*n} ~~\mbox{for all}~~ m,n \in \mathbb{Z}_+.
\]
Therefore, 
\[
V^mW^n \clm= V^mW^nW^{*n}V^{*m}\clm= W^nV^mV^{*m}W^{*n}\clm= W^nV^m \clm
\]
for all $m,n \in \mathbb{Z}_+$.
\end{proof}

\begin{lem}\label{Invariant-subspace-with-equal-range}
Let $(V, W)$ be a pair of isometries on $\clh$ with $\clr( V^m W) = \clr(W V^m )$ for all $m \in \mathbb{Z}_+$. If $\clh = \clh_u \oplus \clh_s$ is the {\textit{Wold decomposition}} for the isometry $V$, then $\clh_u$ is invariant under $W$.
\end{lem}

\begin{proof}
Suppose $\mathcal{H}=\mathcal{H}_{u} \oplus \mathcal{H}_{s}$
is the {\textit{Wold decomposition}} for the isometry $V$, where  $\clh_u= 
\dsp \mathop\cap_{m=0}^{\infty} V^m \clh$. Since $W$ is an isometry 
on $\clh$ and $V^mW \clh =WV^m \clh$ for $m \in \mathbb{Z}_+$, we have
\[
W\clh_{u}= W \dsp \big( \mathop\cap_{m=0}^{\infty} V^m \clh \big)
= \dsp \mathop\cap_{m=0}^{\infty} WV^m \clh= \dsp \mathop\cap_{m=0}^{\infty} 
V^mW \clh \subseteq \dsp \mathop\cap_{m=0}^{\infty} V^m \clh
=\clh_u.
\]
It follows that $\clh_u$ is $W$-invariant subspace. 	
\end{proof}

The following is an immediate result of the aforementioned lemma:

\begin{cor}\label{Reducing-subspace-corollary}
Suppose $(V,W)$ is a pair of isometries acting on $\clh$ such that $\clr(WV^m)=\clr( V^mW)$ and $\clr(W^*V^m)=\clr( V^m W^*)$ for all $m \in \mathbb{Z}_+$. If $\clh=\clh_u \oplus \clh_s$ is the {\textit{Wold decomposition}} for $V$, then both the subspaces 
$\clh_u$ and $\clh_s$ reduce $W$.
\end{cor}

The following results will be useful to determine the {\textit{Wold decomposition}}
for tuples of isometries:

\begin{lem}\label{Reducing-subspace-unitary}
Let $(V, W)$ be a pair of isometries on $\clh$ such that $\clr(V^m W^n)=\clr(W^nV^m)$ for all $m,n \in \mathbb{Z}_+$. Suppose $\clm \subseteq \clh$ reduces both the isometries $V$ and $W$. If either $V|_{\clm}$ or $W|_{\clm}$ is unitary, then 
$V^{*m}W^n \clm =W^n V^{*m} \clm$ for all $m,n \in \mathbb{Z}_+$.	
\end{lem}

\begin{proof}
Assume that $(V, W)$ is a pair of isometries on $\clh$ such 
that $\clr(V^m W^n)=\clr(W^nV^m)$ for all $m,n \in \mathbb{Z}_+$. 
Now using Lemma \ref{Equal-Range-Isometry}, we have 
\[
V^mW^nW^{*n}V^{*m}=W^nV^mV^{*m}W^{*n} \mbox{~~for}~m, n \in \mathbb{Z}_+.
\]
If $V|_{\clm}$ is unitary, then for each $m,n \in \mathbb{Z}_+$, we have
\[
W^nV^{*m}\clm= W^nW^{*n} V^{*m} \clm= V^{*m}V^m W^nW^{*n}V^{*m} \clm
=V^{*m} W^nV^mV^{*m}W^{*n}\clm=V^{*m}W^n \clm.
\]
On the other hand if $W|_{\clm}$ is unitary, then clearly 
\[
V^{*m}W^n \clm= V^{*m} \clm= \clm =W^n\clm= W^nV^{*m}\clm 
\mbox{~~~for all}~m, n \in \mathbb{Z}_+.
\]
This completes the proof.
\end{proof}

\begin{lem}\label{Equivalent-condition-Range}
Let $(V, W)$ be a pair of isometries acting on $\clh$ such that $\clr(V^m W^n) =\clr(W^nV^m)$ for all $m,n \in \mathbb{Z}_+$. Then the following are equivalent:
\begin{itemize}
\item[(i)] $ \clr(V^{*m} W^n)= \clr(W^n V^{*m})$ for all $m,n \in \mathbb{Z}_+$.
\item[(ii)] $\clr(W^{*n} V^m)= \clr(V^m W^{*n})$ for all $m,n \in \mathbb{Z}_+$.
\end{itemize}	
\end{lem}

\begin{proof}
$(i) \Rightarrow (ii)$:
Suppose that $V^{*m}W^n\clh= W^n V^{*m}\clh$ for $m,n \in \mathbb{Z}_+$. Then  
$\clr(V^{*m}W^n)$ is a closed subspace of $\clh$, and hence $\clr(W^{*n} V^m )$ is 
also closed. Therefore,
\[
V^m W^{*n}\clh=W^{*n}W^n V^m \clh=W^{*n} V^m W^{n} V^{*m} \clh =W^{*n} V^m  
V^{*m} W^n \clh=W^{*n} V^m \clh.
\]
	
The converse part follows in the same line. 
\end{proof}

The next result is more inclusive than the last one.

\begin{lem}\label{Lemma- Main result}
Let $(V,W)$ be a pair of isometries acting on $\clh$ such that $\clr(V^m W^n)=\clr(W^nV^m)$ and $\clr(V^{*m} W^n)=\clr(W^nV^{*m})$ for all 
$m,n \in \mathbb{Z}_+$. If $\clm \subseteq \clh$ reduces both the 
isometries $V$ and $W$, then $V^{*m} W^n\clm =W^nV^{*m} \clm$
for all $m,n \in \mathbb{Z}_+$.
\end{lem}

\begin{proof}
If $\clm$ is $V,W$ reducing subspace of $\clh$, then by Lemma \ref{Reducing-subspace-with-equal-range}, $V^{m}W^n \clm=W^nV^{m} \clm$. Again $\clm= (\clh_u \cap \clm) \oplus (\clh_s \cap \clm)$, where $\clh=\clh_u \oplus \clh_s$ is the Wold decomposition for $V$. Now Corollary \ref{Reducing-subspace-corollary} yields that $\clh_u \cap \clm $ and $\clh_s \cap \clm $ reduce both the operators $V$ and $W$. Let $\eta \in \clm$. Then $\eta=h_u \oplus h_s$, where $h_u \in \clh_u \cap \clm $ and $h_s \in  \clh_s \cap \clm$. Since $V|_{\clh_u \cap \clm}$ is unitary, it follows that 
$V^{*m} W^n (\clh_u \cap \clm) = W^n V^{*m} (\clh_u \cap \clm )$. On the other hand,
$V|_{\clh_s \cap \clm}$ is shift and hence $h_s =\sum\limits_{k=0}^{\infty}V^k \eta_k$, where $\eta_k \in  \cln(V^*|_{\clh_s \cap \clm}) \subseteq  \clh_s \cap \clm$. Consequently, for each fixed $m,n \in \mathbb{Z}_+$, we have
\begin{align*}
V^{*m}W^n h_s 
= V^{*m}W^n \big(\sum\limits_{k=0}^{\infty}V^k \eta_k\big)
&= V^{*m}W^n \big(\sum\limits_{k=0}^{m-1}V^k \eta_k\big) +
V^{*m} W^n \big(\sum\limits_{k=m}^{\infty} V^k \eta_k \big)  \\
&= \sum\limits_{k=0}^{m-1} V^{*m}W^nV^k \eta_k  + V^{*m} W^n 
\big(\sum\limits_{k=m}^{\infty} V^k \eta_k \big).
\end{align*}
Since $ \clr(V^{*m} W^n)= \clr(W^n V^{*m})$ and $\clr(W^{*n} V^m)= \clr(V^m W^{*n})$ are equivalent by Lemma \ref {Equivalent-condition-Range}, for each $ h\in \clh$ there exists $h' \in \clh$ such that $W^{*n}V^m h = V^m  W^{*n}h'$. Now for each $h \in \clh$ and for $\eta_k \in \cln(V^*|_{\clh_s \cap \clm})$, where $0 \leq k \leq m-1$, we have
\begin{align*}
\langle  V^{*m} W^n V^k \eta_k, h \rangle
&= \langle W^n V^k \eta_k, V^mh \rangle\\
&= \langle V^k  \eta_k, W^{*n}V^m h \rangle \\
&= \langle V^k \eta_k, V^m  W^{*n}h'\rangle \\
&=\langle \eta_k, V^{m-k}  W^{*n}h' \rangle \\
&=\langle V^{*(m-k)}\eta_k,   W^{*n}h' \rangle \\
&=0.
\end{align*}
It follows that $ V^{*m} W^n V^k \eta_k=0$  for $ 0 \leq k \leq m-1$. Hence
$\sum\limits_{k=0}^{m-1} V^{*m}W^nV^k \eta_k =0$.

Now $V^m, W^n$ are isometries on $\clh$ and $V^mW^n\clh= W^nV^m \clh$ for each 
$m,n \in \mathbb{Z}_+$. Using Lemma \ref{Equal-Range-Isometry}, we have  
$V^m W^n W^{*n}V^{*m}=W^n V^mV^{*m}W^{*n}$. Then
\begin{align*}
W^{*n}V^{*m} & = W^{*n} V^{*m} W^n V^mV^{*m}W^{*n}\\
\mbox{and~~} W^n V^m   &= V^m W^n W^{*n}V^{*m}W^nV^m.
\end{align*}
Now using this, we have
\begin{align*}
W^n V^{*m} (\clh_s \cap \clm) & =W^n W^{*n}V^{*m}  (\clh_s \cap \clm)\\
&= W^n W^{*n} V^{*m}W^n V^m V^{*m}W^{*n} (\clh_s \cap \clm)\\
& = W^nW^{*n}V^{*m} W^n V^m V^{*m}  (\clh_s \cap \clm). 
\end{align*}
Thus for each $ h_s \in  \clh_s \cap \clm$, there exists some $h_s' \in \clh_s \cap \clm$ such that 
\[
W^nW^{*n}V^{*m}W^n V^m V^{*m} h_s= W ^nV^{*m} h_s'.
\]
Now		
\begin{align*}	
V^{*m}W^n h_s=V^{*m} W^n \big(\sum\limits_{k=m}^{\infty} V^k \eta_k \big)
&=V^{*m} W^n V^m \big(\sum\limits_{k=m}^{\infty} V^{k-m} \eta_k \big)\\
&= V^{*m} V^m W^n W^{*n} V^{*m}W^n V^m \big(\sum\limits_{k=m}^{\infty} V^{k-m} 
\eta_k \big) \\
&= W^nW^{*n}V^{*m}W^n V^m V^{*m}\big(\sum\limits_{k=m}^{\infty} V^k \eta_k \big) \\
&=  W^nW^{*n}V^{*m}W^n V^m V^{*m}(h_s - \sum\limits_{k=0}^{m-1}V^k \eta_k ) \\
&= W^nV^{*m} h_s'
\end{align*} 	
as $\eta_k \in \cln(V^*|_{\clh_s \cap \clm})$. Therefore,
\[
V^{*m} W^n \eta= V^{*m} W^n h_u + V^{*m} W^n h_s= W^nV^{*m} h_u' + W^nV^{*m} h_s'=W^nV^{*m} \eta',
\]
where $\eta'= h_u' \oplus h_s' \in \clm$. Hence $V^{*m}W^n \clm=W^nV^{*m} \clm$
for $m,n \in \mathbb{Z}_+$. This completes the proof. 
\end{proof}

Now we shall give some concrete examples of pairs of isometries
$(V_1, V_2)$ on Hilbert spaces satisfying certain properties.

\begin{ex}\label{Example-non commuting but equal range}
Let $(V_1, V_2)$ be a pair of isometries on a Hilbert space 
$\ell^2(\mathbb{Z}_+)$ defined by
\begin{align*}
V_1(x_0,x_1,x_2, \dots) &= (0, x_0,x_1,  \dots),\\
\text{and} \quad V_2(x_0,x_1,x_2, \dots) &=(0, \alpha_0{x_0},\alpha_1 {x_1},\dots)
\end{align*}
where $(x_0, x_1, x_2,\dots) \in \ell^2(\mathbb{Z}_+)$ and $|\alpha_i|=1$ 
for all $i \in \mathbb{Z}_+$. Then for each $(x_0,x_1,x_2, \dots) 
\in \ell^2(\mathbb{Z}_+)$ 
\begin{align*}
V_1V_2(x_0,x_1,x_2, \dots)&= V_1(0, \alpha_0{x_0},\alpha_1 {x_1},\dots) 
=(0,0,  \alpha_0{x_0},\alpha_1 {x_1},\dots), \\
\text{and} \quad  V_2V_1(x_0,x_1,x_2, \dots)&= V_2(0,x_0,x_1,x_2, \dots)
= (0,0, \alpha_1{x_0},\alpha_2 {x_1},\dots).
\end{align*}		
Therefore, the pair of isometries $(V_1,V_2)$ is neither commuting nor twisted.
Again 
\begin{align*}
V_1V_2(x_0,x_1,x_2, \dots)&	=(0,0,  \alpha_0{x_0},\alpha_1 {x_1},\dots) \\
&= V_2V_1(  \frac{\alpha_0}{\alpha_1}{x_0}, \frac{\alpha_1}{\alpha_2} {x_1},\dots) \\
&= V_2V_1(y_0,y_1,y_2, \dots),  
\end{align*}
where $(y_0,y_1,y_2, \dots)= (\frac{\alpha_0}{\alpha_1}{x_0}, \frac{\alpha_1}{\alpha_2}
{x_1},\dots) \in \ell^2(\mathbb{Z}_+)$, that means, $\clr(V_1V_2)=\clr(V_2V_1)$. 
Also 
\begin{align*}
V_2V_1^*(x_0,x_1,x_2,\dots)&=(0, \alpha_0x_1, \alpha_1x_2,\dots),\\
\mbox{and~~} V_1^*V_2((y_0,y_1,y_2,\dots) & = (\alpha_0y_0, \alpha_1y_1,\dots).
\end{align*}
Hence the pair of isometries $(V_1,V_2)$ on $\ell^2(\mathbb{Z}_+)$ is neither 
doubly commuting nor doubly twisted. However, $\clr(V_1V_2)=\clr(V_2V_1)$ but 
$\clr(V_1^*V_2)\neq \clr(V_2V_1^*)$. 
\end{ex}

\begin{ex}\label{Example3-proper}
Let $H^2(\mathbb{D})$ denotes the Hardy space over the open unit disc 
$\mathbb{D}$. The Hardy space over the bidisc $\mathbb{D}^2$, denoted by $H^2(\mathbb{D}^2)$, can be identified with $H^2(\mathbb{D}) \otimes  H^2(\mathbb{D})$ through the canonical unitary $z^{n_1} \otimes z^{n_2}\mapsto z_1^{n_1} z_2^{n_2}$ for $(n_1, n_2) \in \mathbb{Z}_+^2$. The multiplication operator $M_{z_i}$ on $H^2(\mathbb{D}^2)$ by the co-ordinate function $z_i$ on $\mathbb{D}$ is defined by $M_{z_i} f= z_i f$ for $f \in H^2(\mathbb{D}^2)$ and $i=1,2$. 
		
For each sequence $\bm{r} =(r_n) \subset S^1$, define a map $ \Delta[{\bm{r}}]$ on $H^2(\mathbb{D})$ as
\[
\Delta[{\bm{r}}] z^n= {r_n} z^n \hspace{1cm}  (n \in \mathbb{Z}_+). 
\]
Then $\Delta[{\bm{r}}]$ is a unitary diagonal operator. Now define $V_1$ and $V_2$ on $H^2(\mathbb{D}^2)$ as 
\[
V_1=  M_{z} \otimes I_{H^2(\mathbb{D})} 
\ \ \ \text{and} \quad 
V_2= \Delta[{\bm{r}}] \otimes  M_{z}. 
\] 
Clearly, $(V_1,V_2)$ is a pair of isometries on $H^2(\mathbb{D}^2)$. 
Then for $p, q \in \mathbb{Z}_+$,
\[
V_1^p=  M_{z}^p \otimes I_{H^2(\mathbb{D})}
\ \ \ \text{and} \quad 
V_2^q=  \Delta[{\bm{r}}]^q \otimes  M_{z}^q. 
\] 
Now
\begin{align*}
V_1^pV_2^q(z_1^{n_1} z_2^{n_2})
&=  ( M_{z}^p \Delta[{\bm{r}}]^q \otimes M_{z}^q)(z^{n_1} \otimes z^{n_2})\\
&= r_{n_1}^q z^{n_1+p} \otimes  z^{n_2+q} \\
&=\frac{r_{n_1}^q}{r_{n_1+p}^q}( r_{n_1+p}^q z^{n_1+p} \otimes  z^{n_2+q})\\
&=(\Delta[{\bm{r}}]^q M_{z}^p  \otimes  M_{z}^q )\big(\frac{r_{n_1}^q}{r_{n_1+p}^q} 
(z^{n_1} \otimes z^{n_2})\big)\\
&= V_2^q V_1^p( \frac{r_{n_1}^q}{r_{n_1+p}^q} z_1^{n_1} z_2^{n_2}).
\end{align*}
Also
\begin{align*}
V_1^{*p}V_2^q(z_1^{n_1} z_2^{n_2})&= (M_z^{*p}\Delta[{\bm{r}}]^q \otimes  M_{z}^{q}) 
(z^{n_1} \otimes z^{n_2})\\
&=\begin{cases} {r}_{n_1}^q z^{n_1 -p} \otimes  z^{n_2 + q} & 
\text{if $n_1 \geq p$},\\
0 & \text{if $n_1 < p$}
\end{cases} \\
&=\begin{cases}
\frac{{r}_{n_1}^q}{{r}_{n_1 -p}^q} ( {r}_{n_1 -p}^q  z^{n_1- p} \otimes  
z^{n_2 + q } ) & \text{if $n_1 \geq p$},\\
0 & \text{if $n_1 < p$} \\
\end{cases} \\
&= (\Delta[{\bm{r}}]^q M_{z}^{*p}  \otimes  M_{z}^{q}) 
\big (\frac{{r}_{n_1}^q}{{r}_{n_1 -p}^q} ( z^{n_1} \otimes z^{n_2})\big)\\
&=V_2^qV_1^{*p}(\frac{{r}_{n_1}^q}{{r}_{n_1 -p}^q} z_1^{n_1} z_2^{n_2}).
\end{align*}
Therefore, $(V_1,V_2)$ is neither commuting nor twisted, but 
$\clr(V_1^pV_2^q)=\clr(V_2^qV_1^p)$ and $\clr(V_1^{*p}V_2^q)=\clr(V_2^qV_1^{*p})$ for all $p,q \in \mathbb{Z}_+$, that is, $(V_1,V_2)$ is a pair of {\textit{isometries with equal range}}.  
\end{ex}

\section{Structure for $n$-tuple of isometries} 

In this section, we explore some important supplementary results for $n$-tuples of isometries satisfying certain conditions and also establish the structure of decomposition spaces for $n$-tuple of isometries $(V_1, \ldots, V_n)$ on $\clh$ satisfying $\clr(V_i^{m_i} V_j^{m_j})= \clr(V_j^{m_j} V_i^{m_i})$ and $\clr(V_i^{*m_i} V_j^{m_j})= \clr(V_j^{m_j} V_i^{*m_i})$ for $m_i, m_j \in \mathbb{Z}_+$ and $1 \leq i<j \leq n$.

In what follows, the notations below will be fixed for the several variables case. 
We assume $n \geq 2$ is a positive integer. For a given integer $k$ with $1 \leq k \leq n$, we denote the set $\{1,2,\dots,k\}$ by $I_k$. Let $\Lambda =\{i_1< \dots < i_l\} \subseteq I_n $ for $1 \leq l \leq n$, $I_n \backslash \Lambda= \{i_{l+1}< \dots < i_n\}$. The cardinality 
of the set $\Lambda$ is denoted by $|\Lambda|$. We denote by $V_{\Lambda}$ the 
$|\Lambda|$-tuple of isometries $(V_{i_1},\dots, V_{i_l})$ and $\mathbb{Z}_{+}^{|\Lambda|} :=\{ {\bf {m}}=(m_{i_1},\dots,m_{i_l}): m_{i_j}  \in \mathbb{Z}_+, 1 \leq j \leq l \}$. Also $V_{i_1}^{m_{i_1}} \cdots V_{i_l}^{m_{i_l}}$ is denoted by $V_{\Lambda}^ {\bf m}$ for $ {\bf m} \in \mathbb{Z}_+^{\Lambda}$, and $e_j$ denotes the element in $\mathbb{Z}^n_+$ with $1$ in the $j$th position and zero elsewhere. Before proceeding further, we introduce one more notation for the rest of the paper: if $ \Lambda \subseteq I_k$ and $q \notin I_k$, then $\widetilde{\Lambda}$ is denoted by same $\Lambda$ but we will treat $\widetilde{\Lambda}$ as a subset of $I_{k} \cup \{q\}$. 
   
We define $\clw_i=\cln( V_i^*)$ for each $1 \leq i \leq n$ and 
\[
\clw_{\Lambda} (V) = \clw_{\Lambda}:= \bigcap_{i \in \Lambda} \clw_i,
\]
where $\Lambda$ is a non-empty subset of $I_k$ for $1 \leq k \leq n$, 
and $\clw_{\emptyset}=\clh$.

Now we list some important results which will be used throughout the sequel.

\begin{prop}\label{Multi-Index-Reducing-Prop}
Let $(V_1, \ldots, V_n)$ be an $n$-tuple of isometries on $\clh$ 
such that $\clr(V_i V_j)=\clr(V_j V_i)$ and $\clr(V_i^* V_j)=\clr(V_j V_i^*)$ 
for $1 \leq i < j \leq n$. If $\Lambda$ is a subset of $I_n$, then $\clw_{\Lambda}$ 
reduces $V_j$ for each $j \in I_n \backslash \Lambda$.
\end{prop}

\begin{proof}
Since $V_iV_j\clh=V_jV_i\clh$ and $V_iV_j^*\clh=V_j^*V_i\clh$, for every 
$h \in \clh$ there exist $h',h'' \in \clh$ such that $V_jV_ih=V_iV_jh'$ 
and $V_j^*V_ih=V_iV_j^*h''$, respectively.
	
Let $\eta \in \clw_{\Lambda}=\bigcap\limits_{i \in \Lambda} \clw_i$.
For each $ h \in \clh$ and for $ j \in I_n \backslash \Lambda$, 
we have
\begin{align*}
& \langle V_j^*\eta, V_i h \rangle = \langle \eta, V_j V_i h \rangle 
= \langle  \eta, V_iV_j h'\rangle  =  \langle V_i^*\eta, V_j h' \rangle =0 \\
\text{and}   \quad 
& \langle  V_j \eta, V_ih \rangle = \langle  \eta, V_j^* V_i h \rangle  
= \langle  \eta, V_i V_j^*h'' \rangle  =  \langle V_i^*\eta, V_j^*h'' \rangle =0.
\end{align*}
Therefore, for each fixed $j \in I_n \backslash \Lambda$, we obtain
$V_j^*\clw_{\Lambda}, V_j \clw_{\Lambda} \subseteq (V_i \clh)^{\perp}=\clw_i$ 
for all $i \in \Lambda$. Hence $V_j\clw_{\Lambda} \subseteq \clw_{\Lambda}$ 
and  $V_j^*\clw_{\Lambda} \subseteq \clw_{\Lambda}$ for $j \in I_n 
\backslash \Lambda$.  
\end{proof}

\begin{prop}\label{Corollary of intersection form}
Let $(V_1, \ldots, V_n)$ be an $n$-tuple of isometries on $\clh$ such that 
$\clr(V_iV_j)=\clr(V_jV_i)$ and $\clr(V_i^* V_j)=\clr(V_j V_i^*)$ for $1 \leq i < j \leq n$. Then for each subset $\Lambda$ of $I_n$ and $ j \in I_n \backslash \Lambda$,
\[
\clw_{\Lambda} \ominus V_j \clw_{ \Lambda}= \big(\bigcap_{i \in \Lambda} 
\clw_i\big) \cap \clw_j.
\]
\end{prop}

\begin{proof}
Suppose $j \in I_n \backslash \Lambda$. Following Proposition \ref{Multi-Index-Reducing-Prop}, we have $V_j \clw_{\Lambda} \subseteq \clw_{\Lambda}$ and $V_j ^*\clw_{\Lambda} \subseteq \clw_{\Lambda}$. Thus $V_j|_{\clw_{\Lambda}}$ is an isometry on $\clw_{\Lambda}$ and $\cln( V_j^*|_{\clw_\Lambda}) = \clw_{\Lambda} \ominus V_j \clw_{\Lambda} =\clw_{\Lambda} \cap [V_j \clw_{\Lambda}]^{\perp}$. 
Suppose that $\eta \in \clw_{\Lambda} \ominus V_j \clw_{\Lambda}$; then $\eta \in \clw_{\Lambda}$ and $ \eta \perp V_j \clw_{\Lambda}$. It follows $V_j^* \eta \perp \clw_\Lambda$. But $V_j^* \eta \in \clw_\Lambda$ as $\eta \in \clw_\Lambda$ and 
$ V_j^*\clw_\Lambda \subseteq \clw_{\Lambda}$. Therefore, $V_j^* \eta =0$, and hence $\eta \in \clw_j$. As a consequence $\clw_{\Lambda} \ominus V_j \clw_{\Lambda} \subseteq \clw_{\Lambda} \cap \clw_j$.  
	
Now we shall prove the converse inclusion. To do that, let 
$\eta \in \clw_{\Lambda} \cap \clw_j$ for $j \in  I_n \setminus \Lambda$. Then, 
for each $\eta' \in \clw_{\Lambda}$, we get
\[
\langle \eta, V_j \eta' \rangle = \langle V_j^*\eta, \eta' \rangle=0.
\]
Therefore, $\eta \in \clw_\Lambda \cap [V_j\clw_\Lambda]^{\perp}$, that means, 
$\clw_\Lambda \cap \clw_j \subseteq \clw_{\Lambda} \ominus V_j\clw_{\Lambda}$. 
Hence, for each $ j \in I_n \backslash\Lambda$,
\[
\clw_{\Lambda} \ominus V_j\clw_{\Lambda}=\clw_\Lambda \cap \clw_j 
=\big(\bigcap_{i \in \Lambda}\clw_i\big) \cap \clw_j. 
\]	
This finishes the proof.
\end{proof}

The structure of the decomposition spaces for an $n$-tuple of isometries is obtained in the following result, which will be applied in the subsequent sections. It is also important to note that similar result was also observed by firstly Gaspar and Suciu \cite{GS-Wold-decompositions} in the case of commuting tuples of isometries.

\begin{thm}\label{thm-Multi-isometries}
Let $V=(V_1,\ldots, V_n)$ be an $n$-tuple of isometries on $\clh$ 
such that $\clr(V_i^{m_i}V_j^{m_j})=\clr(V_j^{m_j} V_i^{m_i})$ and $\clr(V_i^{*m_i} V_j^{m_j})= \clr(V_j^{m_j} V_i^{*m_i})$ for $m_i,m_j \in \mathbb{Z}_+$, where $1 \leq i<j \leq n$. Then there exist $2^n$ joint $V$-reducing subspaces $\{\clh_{\Lambda}: \Lambda \subseteq I_n \}$ (counting the trivial subspace $\{0\}$) such that 
\begin{align*}
\clh= \bigoplus_{\substack{ \Lambda \subseteq I_n }} \clh_{\Lambda}
\end{align*}
where for each $ \Lambda \subseteq I_n $,
\begin{align}\label{equ-subspace}
\clh_{\Lambda}= \big(\bigcap_{i \in \Lambda} \clh_{si} \big) \cap  \big( \bigcap_{j \in I_n \backslash \Lambda} \clh_{uj}\big),
\end{align}
where $\clh_{si}$ and $\clh_{ui}$ are the shift part and unitary part in the {\textit {Wold decomposition}} for each $V_i$. Also for $\clh_{\Lambda} \neq \{0\}$, $V_i|_{\clh_{\Lambda}}$ is a shift if $i \in  \Lambda$ and $V_j|_{\clh_{\Lambda}}$ is unitary if $j \in I_n \backslash \Lambda$.  Moreover, the above decomposition is unique.
\end{thm}

\begin{proof}
Assume that $V=(V_1,\ldots, V_n)$ is an $n$-tuple of isometries on $\clh$ 
such that $\clr(V_i^{m_i}V_j^{m_j})=\clr(V_j^{m_j} V_i^{m_i})$ and $\clr(V_i^{*m_i} V_j^{m_j})= \clr(V_j^{m_j} V_i^{*m_i})$ for $m_i,m_j \in \mathbb{Z}_+$, where $1 \leq i<j \leq n$. Suppose $\clh=\clh_{ui} \oplus \clh_{si}$ is the {\textit {Wold decomposition}} for each isometry $V_i$ such that $V_i|_{\clh_{ui}}$ is unitary and $V_i|_{\clh_{si}}$ is shift for $i=1,2,\dots,n$. Now the subspaces $\clh_{ui}$ and $\clh_{si}$ reduce the isometries $V_k$ for each $i,k=1,2, \dots, n$ by Corollary \ref{Reducing-subspace-corollary} and Lemma \ref{Equivalent-condition-Range}.

For	each $\Lambda \subseteq I_n$, we define 
\[
\clh_{\Lambda}= \big(\bigcap_{i \in \Lambda} \clh_{si} \big) \cap \big( \bigcap_{j \in I_n \backslash \Lambda} \clh_{uj}\big).
\]
Clearly, $\bigoplus\limits_{\Lambda \subseteq I_n}\clh_{\Lambda} \subseteq \clh$. Thus
it is enough to show that 
\[
\clh= \clh_{u1} \oplus \clh_{s1}\subseteq \bigoplus_{\Lambda \subseteq  I_n} \clh_{\Lambda}.
\]
Since $\clh_{ui}, \clh_{si}$ reduce $V_k$, the space $\clh_{\Lambda}$ reduces $V_k$ for $k=1,2,\dots,n$. Also $\cln (V_k^*|_{\clh_{\Lambda}})= \cln (V_k^*) \cap \clh_{\Lambda}$. Therefore, by Wold decomposition for the isometry $V_k|_{\clh_{\Lambda}}$, we obtain 
\begin{align*}
\clh_{\Lambda}
&=	[\bigcap_{m_k \in \mathbb{Z}_+} V_k^{m_k} \clh_{\Lambda}] \oplus [\bigoplus_{m_k \in \mathbb{Z}_+} V_k^{m_k}(\cln (V_k^*) \cap \clh_{\Lambda})].
\end{align*}
Let $\Lambda \subseteq I_p \subsetneq I_n$ and $k \notin I_p$. It is easy to see that
\begin{align} \label{eq-1}
\bigcap_{m_k \in \mathbb{Z}_+} V_k^{m_k}\clh_{\Lambda} \subseteq (\bigcap_{m_k \in \mathbb{Z}_+} V_k^{m_k}\clh ) \cap \clh_{\Lambda} =  \clh_{uk} \cap \clh_{\Lambda} = \clh_{\tilde{\Lambda}},
\end{align}
where $\tilde{\Lambda}=\Lambda$ treated as a subset of $I_p \cup \{k\}$. Again for each 
$m_k \in \mathbb{Z}_+$ 
\[
V_k^{m_k}(\cln (V_k^*) \cap \clh_{\Lambda}) \subseteq V_k^{m_k} \clh_{\Lambda}\subseteq \clh_{\Lambda}
\]
and 
\[
\bigoplus_{m_k \in \mathbb{Z}_+} V_k^{m_k}(\cln (V_k^*) \cap \clh_{\Lambda}) \subseteq \bigoplus_{m_k \in \mathbb{Z}_+} V_k^{m_k} \cln (V_k^*)=\clh_{sk}.
\]
The above two inclusions give
\begin{align}\label{eq-2}
\bigoplus_{m_k \in \mathbb{Z}_+} V_k^{m_k}(\cln (V_k^*) \cap \clh_{\Lambda}) \subseteq \clh_{\Lambda} \cap \clh_{sk}=\clh_{\tilde{\Lambda} \cup \{k\}}.
\end{align}
Therefore, equations \eqref{eq-1} and \eqref{eq-2} imply that $\clh_{\Lambda} \subseteq \clh_{\tilde{\Lambda} \cup \{k\}} \oplus \clh_{\tilde{\Lambda}}$.

In the similar way, for $\Lambda \subseteq I_{p} \subsetneq I_n$ with $k, l \notin I_p$, we have 
\[
\clh_{\Lambda} \subseteq \clh_{\tilde{\Lambda} \cup \{k\} \cup \{l\}} \oplus \clh_{\tilde{\Lambda} \cup \{k\}} \oplus \clh_{\tilde{\Lambda} \cup \{l\}} \oplus 
\clh_{\tilde{\Lambda}}, 
\]
where $\tilde{\Lambda}=\Lambda $ but we consider it as a subset of $I_p \cup \{k,l\}$. Applying Wold decomposition for $V_1$ on $\clh$, we have
\[
\clh=\clh_{u1} \oplus \clh_{s1}= \clh_{\emptyset} \oplus \clh_{\{1\}},
\]
where $\emptyset, \{1\} \subseteq I_1=\{1\}$. Using the above process repeatedly to 
$\clh_{\emptyset}$ and $\clh_{\{ 1 \} }$, we obtain
\begin{align*}
\clh_{\emptyset} \subseteq \bigoplus_{\Lambda \subseteq \Lambda'} \clh_{\Lambda}
\quad \text{and} \quad \clh_{\{1\}} \subseteq \bigoplus_{\Lambda \subseteq \Lambda'}	\clh_{ \{\tilde{1}\} \cup\Lambda}
\end{align*}
where $\Lambda'=\{2,3,\dots,n\}$ and $\{\tilde{1}\}= \{1\}$ but treated as a subset of $I_n$. Therefore,
\begin{align*}
\clh &= \clh_{\emptyset} \oplus \clh_{\{1\}} \subseteq (\bigoplus_{\Lambda \subseteq \Lambda'} \clh_{\Lambda}) \oplus (\bigoplus_{\Lambda \subseteq \Lambda'}	\clh_{\{\tilde{1}\} \cup \Lambda} )	= \bigoplus_{\Lambda \subseteq I_n} \clh_{\Lambda}.  
\end{align*}
Hence $\clh=\bigoplus\limits_{\Lambda \subseteq I_n} \clh_{\Lambda}$; and $V_j|_{\clh_{\Lambda}}$ is unitary for $ j \in I_n \backslash \Lambda$ and $V_i|_{\clh_{\Lambda}}$ is shift for $i \in \Lambda$. 

Finally, the uniqueness of this decomposition follows from the uniqueness of Wold decomposition for an isometry.
\end{proof}

\begin{rem}
In the preceding result, the structure of decomposition spaces is motivated by the proof of Theorem 2.1 in [12], that characterizes any tuple of isometries admitting {\textit{Wold decomposition}}. We recognize that the structure of the decomposition space in the intersection form presented above is appropriate for any given $n$-tuples of {\textit{isometries with equal range}} which is a bigger class than the tuples of doubly non-commuting isometries and $\clu_n$-twisted isometries.
\end{rem}

\section{Decomposition for pairs of isometries}

In this section, we focus on the representation of decomposition spaces for pairs of {\textit{isometries with equal range}}. It is challenging to express the representation of the decomposition spaces described in Theorem \ref{thm-Multi-isometries}.

Let $(V_1, V_2)$ be a pair of isometries on $\clh$ such that $\clr(V_1^m V_2^n)=\clr(V_2^n V_1^m)$ and $\clr(V_1^{*m} V_2^n)= \clr(V_2^n V_1^{*m})$ for $m,n \in \mathbb{Z}_+$. Also recall that $\clw_i = \cln(V_i^*)$ for $i=1, 2$. Then, a unique decomposition follows from Theorem \ref{thm-Multi-isometries}, and from the equation \eqref{equ-subspace}, the decomposition spaces become
\begin{align*}
\clh_{\emptyset}&=\clh_{u1} \cap \clh_{u2}= \big(\bigcap_{m \in \mathbb{Z}_+} V_1^m \clh \big) \cap \big(\bigcap_{n \in \mathbb{Z}_+} V_2^n \clh\big),\\
\clh_{\{1\}} &= \clh_{s1} \cap \clh_{u2} =  \big(\bigoplus_{m \in \mathbb{Z}_+} V_1^m \clw_1 \big) \cap \big(\bigcap_{n \in \mathbb{Z}_+} V_2^n \clh\big),  \\
\clh_{\{2\}} &=\clh_{u1} \cap \clh_{s2}=\big(\bigcap_{m \in \mathbb{Z}_+} V_1^m \clh \big) \cap  \big(\bigoplus_{n \in \mathbb{Z}_+} V_2^n \clw_2 \big),\\
\clh_{\{1,2\}} &=\clh_{s1} \cap \clh_{s2}=\big(\bigoplus_{m \in \mathbb{Z}_+} V_1^m \clw_1 \big) \cap  \big(\bigoplus_{n \in \mathbb{Z}_+} V_2^n \clw_2 \big) .
\end{align*}
Now our aim is to obtain the complete structure of the orthogonal spaces. Note that
\[
\bigcap_{m,n \in \mathbb{Z}_+} V_1^m V_2^n\clh \subseteq \big(\bigcap\limits_{m \in \mathbb{Z}_+} V_1^m \clh \big) \cap \big(\bigcap\limits_{n \in \mathbb{Z}_+} V_2^n \clh\big).
\]
For the converse part, let $h_{uu} \in  \clh_{\emptyset}=\big(\bigcap\limits_{m \in \mathbb{Z}_+} V_1^m \clh \big) \cap \big(\bigcap\limits_{n \in \mathbb{Z}_+} V_2^n \clh\big)$. Then, for each fixed $n \in \mathbb{Z}_+$,
\begin{align*}
V_2^{*n}h_{uu} \in V_2^{*n}(\bigcap_{m \in \mathbb{Z}_+} V_1^m \clh ) \subseteq 
\bigcap_{m \in \mathbb{Z}_+}V_2^{*n} V_1^m  \clh=\bigcap_{m \in \mathbb{Z}_+} V_1^m V_2^{*n} \clh= \bigcap_{m \in \mathbb{Z}_+} V_1^m \clh. 
\end{align*}
Again $V_2|_{\clh_{\emptyset}}$ is unitary on $\clh_{\emptyset}$. It follows that 
for every fixed $n \in \mathbb{Z}_+$
\[
h_{uu} = V_2^n V_2^{*n}h_{uu} \in \bigcap_{m \in \mathbb{Z}_+} V_2^nV_1^m \clh =\bigcap_{m \in \mathbb{Z}_+} V_1^m V_2^n \clh. 
\]
Hence
\begin{equation}\label{eq-unitary pair}
\clh_{\emptyset}= \big(\bigcap_{m \in \mathbb{Z}_+} V_1^m \clh \big) \cap \big(\bigcap_{n \in \mathbb{Z}_+} V_2^n \clh\big)= \bigcap_{m,n\in \mathbb{Z}_+} V_1^mV_2^n \clh.
\end{equation}
Now $\bigcap\limits_{n \in \mathbb{Z}_+} V_2^n \clw_1 \subseteq \clw_1$ and 
$\bigoplus\limits_{m \in \mathbb{Z}_+} V_1^m \clw_1 \subseteq \clh$ implies that 
$\bigoplus\limits_{m \in \mathbb{Z}_+} V_1^m(\bigcap\limits_{n \in \mathbb{Z}_+} V_2^n \clw_1) \subseteq \bigoplus\limits_{m \in \mathbb{Z}_+} V_1^m\clw_1$ and
$\bigcap\limits_{n \in \mathbb{Z}_+} V_2^n(\bigoplus\limits_{m \in \mathbb{Z}_+} V_1^m \clw_1) \subseteq \bigcap\limits_{n \in \mathbb{Z}_+} V_2^n \clh$. But
$\bigoplus\limits_{m \in \mathbb{Z}_+} V_1^m(\bigcap\limits_{n \in \mathbb{Z}_+} V_2^n \clw_1) = \bigcap\limits_{n \in \mathbb{Z}_+} V_2^n(\bigoplus\limits_{m \in \mathbb{Z}_+} V_1^m \clw_1)$ and hence
\[
\bigoplus_{m \in \mathbb{Z}_+} V_1^m(\bigcap_{n \in \mathbb{Z}_+} V_2^n \clw_1) \subseteq \big(\bigoplus_{m \in \mathbb{Z}_+} V_1^m \clw_1 \big) \cap \big(\bigcap_{n \in \mathbb{Z}_+} V_2^n \clh\big)  
=\clh_{\{1\}}.
\]	
Conversely, let $h_{su} \in \clh_{\{1\}}= \big(\bigoplus\limits_{m \in \mathbb{Z}_+} V_1^m \clw_1 \big) \cap \big(\bigcap\limits_{n \in \mathbb{Z}_+} V_2^n \clh\big)  $. Then
\[
V_2^{*n}h_{su} \in V_2^{*n}(\bigoplus\limits_{m \in \mathbb{Z}_+} V_1^m \clw_1 \big).
\] 
Now using the conditions $V_1^m V_2^n\clh =V_2^nV_1^m \clh$ and $V_1^{*m} V_2^n \clh= V_2^n V_1^{*m}\clh$ and Lemma \ref{Equivalent-condition-Range}, for each fixed $m, n \in \mathbb{Z}_+$, we have 
\[
V_2^{*n} V_1^m \clw_1 = V_2^{*n} V_1^m (\clh \ominus V_1\clh) 
= V_1^m V_2^{*n} \clh \ominus V_1^{m+1} V_2^{*n}\clh = V_1^m (\clh \ominus V_1\clh) = V_1^m \clw_1.
\]
Therefore, 
\[
V_2^{*n} h_{su} \in \bigoplus\limits_{m \in \mathbb{Z}_+}  V_1^m \clw_1.
\]
Since $V_2|_{\clh_{\{1\}}}$ is unitary, for each fixed $n \in \mathbb{Z}_+$, we get
\begin{align*}
h_{su}= V_2^n V_2^{*n}h_{su} \in \bigoplus\limits_{m \in \mathbb{Z}_+} V_2^n V_1^m \clw_1 =\bigoplus\limits_{m \in \mathbb{Z}_+} V_1^m V_2^n \clw_1.
\end{align*}
Therefore, 
\begin{align}
\clh_{\{1\}} = \big(\bigoplus_{m \in \mathbb{Z}_+} V_1^m \clw_1 \big) \cap \big(\bigcap_{n \in \mathbb{Z}_+} V_2^n \clh\big) 
&= \bigcap_{n \in \mathbb{Z}_+} \big(\bigoplus\limits_{m \in \mathbb{Z}_+} 
V_1^m V_2^n \clw_1 \big)\nonumber \\ 
&=	\bigoplus\limits_{m \in \mathbb{Z}_+} V_1^m  \big(\bigcap_{n \in \mathbb{Z}_+} V_2^n \clw_1 \big).
\end{align}
Similarly, 
\begin{equation}
\clh_{\{2\}} =\big(\bigcap_{m \in \mathbb{Z}_+} V_1^m \clh \big) \cap  \big(\bigoplus_{n \in \mathbb{Z}_+} V_2^n \clw_2 \big)=	\bigoplus\limits_{n \in \mathbb{Z}_+} V_2^n  \big(\bigcap_{m \in \mathbb{Z}_+} V_1^m \clw_2 \big).
\end{equation}
Finally, it is easy to see that
\[
\bigoplus_{m,n \in \mathbb{Z}_+} V_1^m V_2^n (\clw_1 \cap \clw_2) \subseteq \big(\bigoplus_{m \in \mathbb{Z}_+} V_1^m \clw_1 \big) \cap  \big(\bigoplus_{n \in \mathbb{Z}_+} V_2^n \clw_2 \big)=\clh_{\{1,2\}}.
\]
To prove the reverse direction, let $h_{ss} \in \clh_{\{1,2\}}= \dsp \big(\bigoplus_{m \in \mathbb{Z}_+} V_1^m \clw_1 \big) \cap  \big(\bigoplus_{n \in \mathbb{Z}_+} V_2^n \clw_2 \big)$. Then 
\[
h_{ss} \in \bigoplus_{m \in \mathbb{Z}_+} V_1^m \clw_1.
\]
Now, using the relation $(I_{\clh}-V_2^n V_2^{*n})V_1^m\clw_1= V_1^m(I_{\clh}-V_2^n V_2^{*n}) \clw_1$ for each fixed $n \in \mathbb{Z}_+$, we obtain
\begin{align*}
(I_{\clh}-V_2^n V_2^{*n})h_{ss} \in (I_{\clh}-V_2^n V_2^{*n}) \Big(\bigoplus_{m \in \mathbb{Z}_+} V_1^m \clw_1 \Big)
&= \bigoplus_{m \in \mathbb{Z}_+}(I_{\clh}-V_2^n V_2^{*n}) V_1^m \clw_1 \\
&= \bigoplus_{m \in \mathbb{Z}_+}V_1^m (I_{\clh}-V_2^n V_2^{*n}) \clw_1.
\end{align*}
Again $V_2|_{\clw_1}$ is an isometry, and thus for each fixed $n \in \mathbb{Z}_{+}$ 
\begin{align*}
(I_{\clh}-V_2^{n}V_2^{*n})\clw_1
&= \clw_1 \ominus V_2^{ n} \clw_1 \\
&=[\clw_1 \ominus V_2 \clw_1] \oplus \cdots \oplus V_2^{ n-1} [\clw_1 \ominus V_2 \clw_1]\\
&= (\clw_1 \cap \clw_2) \oplus \cdots \oplus V_2^{ n-1}( \clw_1 \cap \clw_2) \\
&= \bigoplus_{k=0}^{ n-1}  V_2^{k}( \clw_1 \cap \clw_2),
\end{align*}	
as Corollary \ref{Corollary of intersection form} gives $\clw_1 \ominus V_2 \clw_1= \clw_1 \cap \clw_2$. Consequently,
\[
(I_{\clh}-V_2^n V_2^{*n})h_{ss}  \in \bigoplus_{m=0}^{\infty} V_1^m \Big(\bigoplus_{k=0}^{ n-1} V_2^{k}( \clw_1 \cap \clw_2)\Big).
\]
As $\lim\limits_{n \rightarrow\infty}V_2^{*n}|_{\clh_{\{1, 2\}}}=0$ in the strong operator topology, letting $n \rightarrow \infty$, we obtain
\begin{align*}
h_{ss} \in \bigoplus_{m=0}^{\infty} V_1^m \Big(\bigoplus_{n=0}^{ \infty}  V_2^{n}( \clw_1 \cap \clw_2)\Big)=\bigoplus_{m,n=0}^{\infty} V_1^m V_2^n (\clw_1 \cap \clw_2).
\end{align*}
Hence,
\begin{equation}\label{eq-shift pair}
\clh_{\{1,2\}}= \big(\bigoplus_{m \in \mathbb{Z}_+} V_1^m \clw_1 \big) \cap  \big(\bigoplus_{n \in \mathbb{Z}_+} V_2^n \clw_2 \big) =\bigoplus_{m,n \in \mathbb{Z}_+} V_1^m V_2^n (\clw_1 \cap \clw_2).
\end{equation}

To summarize the above, we have the following result:
	
\begin{thm}\label{Wold-type-pairs-isomteries}
Let $\left( V_1, V_2 \right)$ be a pair of isometries on a Hilbert space $\clh$ such that $\clr(V_1^m V_2^n)= \clr(V_2^n V_1^m)$ and $\clr(V_1^{*m} V_2^n)= \clr(V_2^n V_1^{*m})$ for $m,n \in \mathbb{Z}_{+}$. Then there is a unique decomposition 
\[
\mathcal{H}=\mathcal{H}_{uu} \oplus \mathcal{H}_{us} \oplus \mathcal{H}_{su} 
\oplus \mathcal{H}_{ss},
\]
where $\mathcal{H}_{uu}, \mathcal{H}_{us}, \mathcal{H}_{su},$ and 
$\mathcal{H}_{ss}$ are the subspaces reducing $V_1,  V_2$ such that
\begin{itemize}
\item	$V_1|_{\mathcal{H}_{uu}}, V_2|_{\mathcal{H}_{uu}}$ are unitary operators,
		
\item   $V_1|_{\mathcal{H}_{us}}$ is unitary, $V_2|_{\mathcal{H}_{us}}$ is a shift,
		
\item  $V_1|_{\mathcal{H}_{su}}$ is shift, $V_2|_{\mathcal{H}_{su}}$ is unitary, and 
		
\item $V_1|_{\mathcal{H}_{ss}}$, $V_2|_{\mathcal{H}_{ss}}$ are shifts. 
\end{itemize} 	
Moreover,
\begin{align*}
&\clh_{uu} = \bigcap_{ m, n =0}^{\infty} V_1^{ m} V_2^{ n} \clh,  
&\clh_{us}= \bigoplus_{ n=0}^{\infty} V_2^{n}
	\Big (\bigcap_{ m=0}^{\infty} V_1^{ m}   \cln(V_2^*) \Big),\\
& \clh_{su}=\bigoplus_{ m=0}^{\infty}  V_1^{ m} 
	\Big( \bigcap_{ n=0}^{\infty} V_2^{ n}  \cln (V_1^*)  \Big), 
&\clh_{ss}= \bigoplus_{ m, n=0}^{\infty}  V_1^{ m}  V_2^{ n}  
	\Big( \cln(V_1^*) \cap \cln(V_2^*) \Big).
\end{align*}
\end{thm}

\begin{rem}
Note that the class of pairs of isometries $(V_1, V_2)$ with equal range, i.e.,
$\clr(V_1^m V_2^n)= \clr(V_2^n V_1^m)$ and $\clr(V_1^{*m} V_2^n)= \clr(V_2^n V_1^{*m})$ for $m,n \in \mathbb{Z}_{+}$ contains the class of $\clu_2$-twisted isometries (in particular, class of pairs of doubly commuting isometries). However, the orthogonal decomposition spaces for the pairs of {\textit{isometries with equal range}} are the same as the orthogonal decomposition spaces for the pairs of doubly commuting isometries, which were obtained by S\l{}o\'{c}inski in \cite{SLOCINSKI-WOLD}. Moreover, the decomposition is unique in both cases.
\end{rem}

\section{Decomposition for $n$-tuples of isometries}

This section deals with {\it{Wold decomposition}} for $n$-tuples of {\textit{isometries with equal range}}. Here, we find the explicit representation of the orthogonal decomposition spaces.

We begin with a few significant findings that will be helpful in the follow-up.

\begin{lem}\label{Important Lemma for n-tuples}
Let $V=(V_1,\dots, V_n)$ be an n-tuple of isometries on $\clh$ such that $\clr(V_{i}^{m_{i}}V_{j}^{m_{j}})=\clr(V_{j}^{m_{j}} V_{i}^{m_{i}})$ and $\clr(V_{i}^{*m_{i}}V_{j}^{m_{j}})= \clr(V_{j}^{m_{j}} V_{i}^{*m_{i}})$ for $m_{i},m_{j} \in \mathbb{Z}_+$, where $1 \leq i< j \leq n$. Then
\begin{align*}
&V_{i_1}^{m_{i_1}} V_{i_2}^{m_{i_2}}\cdots V_{i_{n}}^{m_{i_{n}}}\clh
=V_{i_2}^{m_{i_2}}V_{i_3}^{m_{i_3}} \cdots V_{i_{n}}^{m_{i_{n}}} V_{i_1}^{m_{i_1}}\clh\\	
\text{and}  \quad \quad
&V_{i_1}^{*m_{i_1}} V_{i_2}^{m_{i_2}}\cdots V_{i_{n}}^{m_{i_{n}}}\clh
=V_{i_2}^{m_{i_2}}V_{i_3}^{m_{i_3}} \cdots V_{i_{n}}^{m_{i_{n}}}
V_{i_1}^{*m_{i_1}}\clh,
\end{align*}	
where $m_{i_j} \in \mathbb{Z}_+$ for $i_j \in \{1,2, \dots,n\}$.
\end{lem}

\begin{proof}
Consider $\Lambda = \{i_1 <  \dots< i_l\} $ for $2 \leq l \leq n$ and $i_j \in \{1,2, \dots,n\}$. Then  $(V_{i_p},V_{i_q})$ is any pair of isometries on $\clh$ such that $V_{i_p}^{m_{i_p}}V_{i_q}^{m_{i_q}}\clh=V_{i_q}^{m_{i_q}} V_{i_p}^{m_{i_p}}\clh$ and 
$V_{i_p}^{*m_{i_p}} V_{i_q}^{m_{i_q}}\clh=V_{i_q}^{m_{i_q}}V_{i_p}^{*m_{i_p}}\clh$,
where $m_{i_p}, m_{i_q} \in \mathbb{Z}_+$ and $i_p,i_q \in \{1,2, \dots, n\}$ for $1 \leq p < q \leq n$. Thus, the statement is trivial for $n=2$.

We will prove the statement for $n \geq 3$ by mathematical induction.  Now, the given conditions $V_{i_p}^{m_{i_p}}V_{i_3}^{m_{i_3}}\clh= V_{i_3}^{m_{i_3}}V_{i_p}^{m_{i_p}}\clh$ and $V_{i_p}^{*m_{i_p}}V_{i_3}^{m_{i_3}}\clh= V_{i_3}^{m_{i_3}}V_{i_p}^{*m_{i_p}}\clh$ infer that $ V_{i_3}^{m_{i_3}}\clh$ reduces $V_{i_p}$ for $p=1,2$. Therefore, by Lemma \ref{Reducing-subspace-with-equal-range} and Lemma \ref{Lemma- Main result}, we have  
\[
V_{i_1}^{m_{i_1}}V_{i_2}^{m_{i_2}} V_{i_3}^{m_{i_3}}\clh=V_{i_2}^{m_{i_2}}V_{i_1}^{m_{i_1}} V_{i_3}^{m_{i_3}}\clh= V_{i_2}^{m_{i_2}}V_{i_3}^{m_{i_3}} V_{i_1}^{m_{i_1}} \clh
\]
and 
\[
V_{i_1}^{*m_{i_1}} V_{i_2}^{m_{i_2}}  V_{i_3}^{m_{i_3}}\clh= V_{i_2}^{m_{i_2}} 	V_{i_1}^{*m_{i_1}} V_{i_3}^{m_{i_3}}\clh=V_{i_2}^{m_{i_2}}  V_{i_3}^{m_{i_3}}	V_{i_1}^{*m_{i_1}}\clh
\]
where $i_1, i_2, i_3$ are distinct elements in $\{1,2, \dots, n\}$, respectively. 
Hence the statement is true for $n=3$.

Suppose the above statement is true for any $k$-tuple $(V_{i_1}, \ldots, V_{i_k})$, $k <n$, of isometries on $\clh$ such that $V_{i_p}^{m_{i_p}}V_{i_q}^{m_{i_q}}\clh=V_{i_q}^{m_{i_q}} V_{i_p}^{m_{i_p}}\clh$ and 
$V_{i_p}^{*m_{i_p}} V_{i_q}^{m_{i_q}}\clh=V_{i_q}^{m_{i_q}}V_{i_p}^{*m_{i_p}}\clh$,
where $m_{i_p}, m_{i_q} \in \mathbb{Z}_+$ and $i_p,i_q \in \{1,2, \dots, n\}$ for $1 \leq p < q \leq n$. Then  
\begin{align*} 
V_{i_1}^{m_{i_1}} V_{i_2}^{m_{i_2}}\cdots V_{i_k}^{m_{i_k}}\clh
=V_{i_2}^{m_{i_2}}V_{i_3}^{m_{i_3}} \cdots V_{i_k}^{m_{i_k}} V_{i_1}^{m_{i_1}}\clh
\end{align*}
and 
\begin{align*}
V_{i_1}^{*m_{i_1}} V_{i_2}^{m_{i_2}}\cdots V_{i_k}^{m_{i_k}}\clh
=V_{i_2}^{m_{i_2}}V_{i_3}^{m_{i_3}} \cdots V_{i_k}^{m_{i_k}} V_{i_1}^{*m_{i_1}}\clh
\end{align*}	
for $m_{i_j} \in \mathbb{Z}_+$, where $1\leq j \leq k <n$. 
Since the statement is true for any $k$-tuple of isometries with $k<n$, replacing $V_{i_1}$ by $V_{i_{k+1}}$ we have
\[
V_{i_{k+1}}^{m_{i_{k+1}}} V_{i_2}^{m_{i_2}}\dots V_{i_k}^{m_{i_k}}\clh
=V_{i_2}^{m_{i_2}}V_{i_3}^{m_{i_3}} \cdots V_{i_k}^{m_{i_k}} V_{i_{k+1}}^{m_{i_{k+1}}}\clh
\]
and 
\[
V_{i_{k+1}}^{*m_{i_{k+1}}} V_{i_2}^{m_{i_2}}\dots V_{i_k}^{m_{i_k}}\clh
=V_{i_2}^{m_{i_2}}V_{i_3}^{m_{i_3}} \cdots V_{i_k}^{m_{i_k}}V_{i_{k+1}}^{*m_{i_{k+1}}}  \clh.
\]
Set $\clh'= V_{i_2}^{m_{i_2}}V_{i_3}^{m_{i_3}} \cdots V_{i_k}^{m_{i_k}} \clh$. Then the above relations imply that $\clh'$ reduces $V_{i_1}, V_{i_{k+1}}$ and employing Lemma \ref{Reducing-subspace-with-equal-range}, we get $V_{i_1}^{m_{i_1}} V_{i_{k+1}}^{m_{i_{k+1}}}\clh'=V_{i_{k+1}}^{m_{i_{k+1}}}V_{i_1}^{m_{i_1}}\clh'$. In the same lines of argument the subspace $ V_{i_3}^{m_{i_3}}\cdots V_{i_k}^{m_{i_k}} V_{i_1}^{m_{i_1}}\clh $ reduces the isometries $V_{i_2}$ and $V_{i_{k+1}}$. Therefore,
\begin{align*}
V_{i_1}^{m_{i_1}} V_{i_2}^{m_{i_2}}\cdots V_{i_k}^{m_{i_k}}V_{i_{k+1}}^{m_{i_{k+1}}}\clh
&=	V_{i_1}^{m_{i_1}}V_{i_{k+1}}^{m_{i_{k+1}}} V_{i_2}^{m_{i_2}}\cdots V_{i_k}^{m_{i_k}}\clh \\
&= V_{i_{k+1}}^{m_{i_{k+1}}}	V_{i_1}^{m_{i_1}} (V_{i_2}^{m_{i_2}}\cdots V_{i_k}^{m_{i_k}}\clh )  \\
&=V_{i_{k+1}}^{m_{i_{k+1}}} V_{i_2}^{m_{i_2}} (V_{i_3}^{m_{i_3}}\cdots V_{i_k}^{m_{i_k}}	V_{i_1}^{m_{i_1}}\clh)   \\
&= V_{i_2}^{m_{i_2}} V_{i_{k+1}}^{m_{i_{k+1}}} (V_{i_3}^{m_{i_3}} \cdots V_{i_k}^{m_{i_k}}	V_{i_1}^{m_{i_1}}\clh)\\
&=V_{i_2}^{m_{i_2}}V_{i_3}^{m_{i_3}} \cdots V_{i_k}^{m_{i_k}}V_{i_{k+1}}^{m_{i_{k+1}}} V_{i_1}^{m_{i_1}}\clh. 
\end{align*}
On the other hand, using Lemma \ref{Lemma- Main result}, we obtain
\begin{align*}
V_{i_1}^{*m_{i_1}} V_{i_2}^{m_{i_2}}\cdots V_{i_k}^{m_{i_k}} V_{i_{k+1}}^{m_{i_{k+1}}}\clh
&=V_{i_1}^{*m_{i_1}}V_{i_{k+1}}^{m_{i_{k+1}}} V_{i_2}^{m_{i_2}}\cdots V_{i_k}^{m_{i_k}} \clh\\
&=V_{i_{k+1}}^{m_{i_{k+1}}}V_{i_1}^{*m_{i_1}} (V_{i_2}^{m_{i_2}}\cdots V_{i_k}^{m_{i_k}} \clh)\\
&=V_{i_{k+1}}^{m_{i_{k+1}}} V_{i_2}^{m_{i_2}}\cdots V_{i_k}^{m_{i_k}} V_{i_1}^{*m_{i_1}}\clh \\
&=V_{i_2}^{m_{i_2}}\cdots V_{i_k}^{m_{i_k}} V_{i_{k+1}}^{m_{i_{k+1}}} \clh \\
&=V_{i_2}^{m_{i_2}}\cdots V_{i_k}^{m_{i_k}} V_{i_{k+1}}^{m_{i_{k+1}}} V_{i_1}^{*m_{i_1}}\clh.
\end{align*}
Hence the statement is true for $n=k+1$. 
	
This completes the proof. 
\end{proof}

\begin{lem}\label{Main Corollary for n-tuple}
Let $V=(V_1,\dots, V_n)$ be an $n$-tuple of isometries on $\clh$ such that 
$\clr(V_i^{m_i}V_j^{m_j})=\clr(V_j^{m_j} V_i^{m_i})$ and $\clr(V_i^{*m_i} V_j^{m_j})=\clr(V_j^{m_j} V_i^{*m_i})$ for $m_i,m_j \in \mathbb{Z}_+$, where $1 \leq i<j \leq n$. If $\Lambda$ is a subset of $I_k$ for $k < n$, then
\begin{enumerate}
\item $V^{\bm q}_{ I_k \backslash \Lambda} V_{k+1}^{p_{k+1}} \clw_{\Lambda \cup \{k+1\}}=	 V_{k+1}^{p_{k+1}} V^{\bm q}_{ I_k \backslash \Lambda}  \clw_{\Lambda \cup \{k+1\}}$.
\item $V^{\bm q}_{ I_k \backslash \Lambda} V_{\Lambda}^{\bm p} \clw_{\Lambda}=	 V_{\Lambda}^{\bm p} V^{\bm q}_{ I_k \backslash \Lambda}  \clw_{\Lambda}$. 
\item $V^{*\bm q}_{ I_k \backslash \Lambda} V_{\Lambda}^{\bm p} \clw_{\Lambda}=	 V_{\Lambda}^{\bm p} V^{*\bm q}_{ I_k \backslash \Lambda}  \clw_{\Lambda}$.
\item $V^{\bm p}_{\Lambda} V_{k+1}^{p_{k+1}}   \clw_{\Lambda \cup \{k+1\}}=
V_{k+1}^{p_{k+1}} V^{\bm p}_{ \Lambda}  \clw_{\Lambda \cup \{k+1\}}$.		
\end{enumerate}
\end{lem}

\begin{proof} $(1)$
Let $\Lambda =\{1, \dots, l\} \subseteq I_k $ for $1 \leq l \leq k <n$, 
and $I_k \backslash \Lambda= \{l+1, \dots , k\}$. Then, using Proposition \ref{Corollary of intersection form}, we obtain 
\begin{align*}
V^{\bm q}_{ I_k \backslash \Lambda} V_{k+1}^{p_{k+1}} \clw_{\Lambda \cup \{k+1\}} 
&=V_{{ l+1}}^{ q_{ l+1}} \cdots V_{k}^{ q_{k}}  V_{k+1}^{p_{k+1}}  
[\clw_{\Lambda} \ominus V_{k+1}\clw_{\Lambda}] \\
&=V_{{ l+1}}^{ q_{ l+1}} \cdots V_{k}^{ q_{k}}  V_{k+1}^{p_{k+1}} \clw_{\Lambda}
-V_{{ l+1}}^{ q_{ l+1}} \cdots V_{k}^{ q_{k}}  V_{k+1}^{p_{k+1}+1}  \clw_{\Lambda}.
\end{align*}
Applying Lemma \ref{Important Lemma for n-tuples} repeatedly, we obtain 
\[
V_{ l+1}^{q_{ l+1}} V_{ l+2}^{q_{ l+2}} \cdots  V_k^{q_k} V_{k+1}^{p_{k+1}}\clh
=  V_{k+1}^{p_{k+1}}V_{ l+1}^{q_{ l+1}} V_{ l+2}^{q_{ l+2}} \cdots  V_k^{q_k}\clh,
\]
that means $ V_{ l+1}^{q_{ l+1}} V_{ l+2}^{q_{ l+2}} \cdots  V_k^{q_k}$ and $V_{k+1}^{p_{k+1}}$ are isometries on $\clh$ such that 
\[
\clr[( V_{ l+1}^{q_{ l+1}} V_{ l+2}^{q_{ l+2}} \cdots  V_k^{q_k})V_{k+1}^{p_{k+1}}]
=\clr[ V_{k+1}^{p_{k+1}} (V_{ l+1}^{q_{ l+1}} V_{ l+2}^{q_{ l+2}} \cdots  V_k^{q_k})].
\]		
Now Proposition \ref{Multi-Index-Reducing-Prop} infers that $\clw_{\Lambda}$ is 
$V_{ l+1}^{q_{ l+1}}, \dots , V_k^{q_k}, V_{k+1}^{p_{k+1}}$-reducing subspace of 
$\clh$ and hence, by Lemma \ref{Reducing-subspace-with-equal-range} we have
\[
V_{ l+1}^{q_{ l+1}} \cdots  V_k^{q_k} V_{k+1}^{p_{k+1}}\clw_{\Lambda}
= V_{k+1}^{p_{k+1}} V_{ l+1}^{q_{ l+1}} \cdots  V_k^{q_k}  \clw_{\Lambda},
\]
and also 
\begin{align*}
V_{ l+1}^{q_{ l+1}} \cdots  V_k^{q_k} V_{k+1}^{p_{k+1}+1}\clw_{\Lambda}
&= V_{k+1}^{p_{k+1}+1} V_{ l+1}^{q_{ l+1}} \cdots  V_k^{q_k} \clw_{\Lambda} \\
&=V_{k+1}^{p_{k+1}} V_{ l+1}^{q_{ l+1}} \cdots  V_k^{q_k} V_{k+1}\clw_{\Lambda}.
\end{align*}
Hence,
\begin{align*}
V^{\bm q}_{ I_k \backslash \Lambda} V_{k+1}^{p_{k+1}} \clw_{\Lambda \cup \{k+1\}}
&=V_{k+1}^{p_{k+1}} V_{ l+1}^{q_{ l+1}} \cdots  V_k^{q_k}   [\clw_{\Lambda} 
\ominus   V_{k+1}\clw_{\Lambda}]	 \\
&= V_{k+1}^{p_{k+1}} V^{\bm q}_{ I_k \backslash \Lambda}  \clw_{\Lambda \cup \{k+1\}}.
\end{align*}

(2) Note that $\clw_1=\clh \ominus V_1\clh$ and $\clw_1 \cap \clw_2= \clw_1 \ominus V_2 \clw_1= \clh -V_1\clh-V_2\clh+ V_1V_2\clh$.  Let $\Lambda=\{1, \dots, l\} \subseteq I_n$. Then, using the mathematical induction, we can show that 
\[
\clw_{\Lambda}=\clh - \sum_{i}V_i \clh + \sum_{i,j}V_i V_j\clh-\dots +(-1)^lV_1V_2\cdots V_l \clh.
\]
Now using Lemma \ref{Important Lemma for n-tuples}, we obtain
\begin{align*}
V^{\bm q}_{ I_k \backslash \Lambda} V_{\Lambda}^{\bm p} \clw_{\Lambda}
&= V_{\Lambda}^{\bm p} V^{\bm q}_{ I_k \backslash \Lambda}\clw_{\Lambda}.
\end{align*}

(3) Again using Lemma \ref{Important Lemma for n-tuples}, we have in the same line 
\begin{align*}
V^{*\bm q}_{ I_k \backslash \Lambda} V_{\Lambda}^{\bm p} \clw_{\Lambda}
= V_{\Lambda}^{\bm p} V^{*\bm q}_{ I_k \backslash \Lambda}\clw_{\Lambda}.
\end{align*}

(4) Since $\clw_{\Lambda \cup \{k+1\}}=\clw_{\Lambda} \ominus V_{k+1}\clw_{\Lambda} $, relation (2) yields
\begin{align*}
V^{\bm p}_{\Lambda} V_{k+1}^{p_{k+1}}   \clw_{\Lambda \cup \{k+1\}}
&=	V^{\bm p}_{\Lambda} V_{k+1}^{p_{k+1}} \clw_{\Lambda} \ominus 	V^{\bm p}_{\Lambda} V_{k+1}^{p_{k+1}+1}\clw_{\Lambda} \\
&= V_{k+1}^{p_{k+1}} 	V^{\bm p}_{\Lambda} \clw_{\Lambda}\ominus  V_{k+1}^{p_{k+1}} V^{\bm p}_{\Lambda} V_{k+1}\clw_{\Lambda} \\
&=V_{k+1}^{p_{k+1}} 	V^{\bm p}_{\Lambda} \clw_{{\Lambda} \cup \{k+1\}}.
\end{align*} 
\end{proof}

\begin{lem}\label{Corollary for unitary} 
Let $V=(V_1,\dots, V_n)$ be an n-tuple isometries on $\clh$ such that $\clr(V_i^{m_i}V_j^{m_j})=\clr(V_j^{m_j} V_i^{m_i})$ and $\clr(V_i^{*m_i}V_j^{m_j})=\clr(V_j^{m_j} V_i^{*m_i})$ for  $m_i,m_j \in \mathbb{Z}_+$ where $1 \leq i<j \leq n$. Suppose 
$\clh=\clh_{uj} \oplus \clh_{sj}$ is the {\it{Wold decomposition}} for $V_j$ 
such that $\clh_{uj}= \bigcap\limits_{q_j \in \mathbb{Z}_+} V_j^{q_j}\clh$ for 
$j =1, \dots,n$. Then 
\[
\bigcap_{j \in  I_l} \clh_{uj}
= \bigcap_{q_1, \dots, q_l \in \mathbb{Z}_+} V_1^{q_1} \cdots V_l^{q_l} \clh,
\]
where $ 1 \leq j \leq l \leq n$.
\end{lem}

\begin{proof}
Since $(V_1, \dots, V_n)$ is an $n$-tuple of isometries on $\clh$, for $1 \leq j \leq l \leq n$, we obtain 
\[
\bigcap_{q_1, \dots, q_l \in \mathbb{Z}_+} V_1^{q_1} \cdots V_l^{q_l} \clh
\subseteq \bigcap_{j \in I_l}  \big(\bigcap_{q_j \in \mathbb{Z}_+} V_j^{q_j}\clh \big) =\bigcap_{j \in I_l}   \clh_{uj}.
\] 
It suffices to prove that
\begin{align}\label{eq-unitary n tuple}
\bigcap_{j \in I_l}  \big(\bigcap_{q_j \in \mathbb{Z}_+} V_j^{q_j}\clh \big) \subseteq \bigcap_{q_1, \dots, q_l \in \mathbb{Z}_+} V_1^{q_1} \cdots V_l^{q_l} \clh
\end{align}
for $1 \leq j \leq l \leq n$, and we will prove this by mathematical induction. 
If $(V_1,V_2)$ is a pair of isometries on $\clh$ such that $V_1^mV_2^n \clh=V_2^n V_1^m\clh$ and $V_1^{*m} V_2^n \clh=V_2^n V_1^{*m}\clh$ for $m,n \in \mathbb{Z}_+$, the equation \eqref{eq-unitary pair} implies that the above statement is true for $l=2$.

Now suppose the statement is true for any $k$-tuple, $k <l$, of isometries 
$(V_1, \dots, V_k)$ on $\clh$, that means for $1\leq j\leq  k< l$
\begin{align*}
\bigcap_{j \in I_k}  \big(\bigcap_{q_j \in \mathbb{Z}_+} V_j^{q_j}\clh \big) 
\subseteq \bigcap_{q_1, \dots, q_k \in \mathbb{Z}_+} V_1^{q_1} \cdots V_k^{q_k} \clh.
\end{align*}
Let $h \in \bigcap\limits_{j \in I_{k+1}}  \big(\bigcap\limits_{q_j \in \mathbb{Z}_+} V_j^{q_j}\clh)=\bigcap\limits_{j \in I_{k}}  \big(\bigcap\limits_{q_j \in \mathbb{Z}_+} V_j^{q_j}\clh) \cap \big(\bigcap\limits_{q_{k+1} \in \mathbb{Z}_+} V_{k+1}^{q_{k+1}}\clh)$. Then, for each $q_{k+1} \in \mathbb{Z}_+$, we have 
\[
V_{k+1}^{*q_{k+1}}h  \in  \bigcap\limits_{j \in I_{k}} \big(\bigcap\limits_{q_j \in \mathbb{Z}_+} V_j^{q_j}V_{k+1}^{*q_{k+1}} \clh)= \bigcap_{j \in I_k}\big(\bigcap_{q_j \in \mathbb{Z}_+} V_j^{q_j}\clh \big) \subseteq \bigcap_{q_1, \dots, q_k \in \mathbb{Z}_+} V_1^{q_1} \cdots V_k^{q_k} \clh.
\] 
Therefore, using Lemma \ref{Important Lemma for n-tuples}, we get
\[
h \in \bigcap_{q_1, \dots, q_k, q_{k+1} \in \mathbb{Z}_+} V_{k+1}^{q_{k+1}} V_1^{q_1} \cdots V_k^{q_k}  \clh =\bigcap_{q_1, \dots, q_k , q_{k+1}\in \mathbb{Z}_+} V_1^{q_1} \cdots V_k^{q_k} V_{k+1}^{q_{k+1}} \clh.
\]
Hence,  \eqref{eq-unitary n tuple} holds for $l=k+1$ and this finishes the proof.
\end{proof}

\begin{lem}\label{Corollary for shift}
Let $V=(V_1,\dots, V_n)$ be an  $n$-tuple of isometries on $\clh$ such that 
$\clr(V_i^{m_i}V_j^{m_j})=\clr(V_j^{m_j} V_i^{m_i})$ and $\clr(V_i^{*m_i}V_j^{m_j})=\clr(V_j^{m_j} V_i^{*m_i})$ for $m_i,m_j \in \mathbb{Z}_+$, where $1 \leq i<j \leq n$. Suppose that $\clh=\clh_{ui} \oplus \clh_{si}$ is the {\textit{Wold decomposition}} for $V_i$ such that $\clh_{si}= \bigoplus\limits_{p_i \in \mathbb{Z}_+} V_i^{p_i} \clw_i$ for $\ i =1,2, \dots, n$. Then 
\[
\bigcap_{i \in I_l}   \clh_{si}= \bigoplus_{p_1, \dots, p_l \in \mathbb{Z}_+} V_1^{p_1} \cdots V_l^{p_l} (\clw_1 \cap \dots \cap \clw_l),
\]
where $ 1 \leq i \leq l \leq n$.
\end{lem}

\begin{proof}
Suppose $(V_1,V_2)$ is a pair of  isometries on $\clh$ such that $\clr(V_1^m V_2^n)=\clr(V_2^n V_1^m)$ and $\clr(V_1^{*m} V_2^n)=\clr(V_2^n V_1^{*m})$ for $m,n \in \mathbb{Z}_+$. Then the equation \eqref{eq-shift pair} implies that the above statement is true for $l=2$.

Assume that the statement is true for any $k$-tuple $(k <l)$ of isometries 
$(V_1, \dots, V_k)$ on $\clh$. Then, for $1\leq i\leq  k< l$, we have 
\begin{align*}
\bigcap_{i \in I_k}   \big(\bigoplus\limits_{p_i \in \mathbb{Z}_+}V_i^{p_i}\clw_{i} \big) = \bigoplus_{p_1, \dots, p_k \in \mathbb{Z}_+} V_1^{p_1} \cdots V_k^{p_k} (\clw_1 \cap \dots \cap \clw_k).
\end{align*}
To show the above statement is true for $(k+1)$-tuple of isometries $(V_1, \dots, V_{k+1})$ on $\clh$, firstly observe that for any $i \in I_{k+1}$,  
\[
\clw_1 \cap \cdots \cap \clw_{k+1} \subseteq \clw_i.
\]
Now using the fact $V_j\clw_i \subseteq \clw_i$ for $ 1\leq i \neq j \leq k+1$ and  applying Lemma \ref{Main Corollary for n-tuple}, for each $i \in I_{k+1}$, we have 
\[
\bigoplus_{p_1, \dots, p_{k+1}\in \mathbb{Z}_+} V_1^{p_1} \cdots V_{k+1}^{p_{k+1}} (\clw_1 \cap \dots \cap \clw_{k+1}) \subseteq \bigoplus_{p_1, \dots, p_{k+1} \in \mathbb{Z}_+} V_1^{p_1} \cdots V_{k+1}^{p_{k+1}} \clw_i  \subseteq  \bigoplus\limits_{p_i \in \mathbb{Z}_+}V_i^{p_i}\clw_{i}.
\]
It follows that 
\begin{equation*}
\bigoplus_{p_1, \dots, p_{k+1} \in \mathbb{Z}_+} V_1^{p_1} \cdots V_{k+1}^{p_{k+1}} (\clw_1 \cap \dots \cap \clw_{k+1}) \subseteq 	\bigcap_{i \in I_{k+1}}   \big(\bigoplus\limits_{p_i \in \mathbb{Z}_+}V_i^{p_i}\clw_{i} \big).
\end{equation*}

For the reverse inclusion, let 
\begin{equation}\label{eq-rev inclusion}
h \in 	\bigcap_{i \in I_{k+1}}   \big(\bigoplus\limits_{p_i \in \mathbb{Z}_+}V_i^{p_i}\clw_{i} \big)= \bigoplus_{p_1, \dots, p_k \in \mathbb{Z}_+} V_1^{p_1} \cdots V_k^{p_k} (\clw_1 \cap \dots \cap \clw_k) \cap \big(\bigoplus\limits_{p_{k+1} \in \mathbb{Z}_+}V_{k+1}^{p_{k+1}}\clw_{i} \big).
\end{equation} 
Define $\clw_{K}={{\bigcap\limits_{i=1}^{k} }\clw_i}$. Then, by Proposition \ref{Multi-Index-Reducing-Prop}, $V_{k+1}|_{\clw_{K}}$ is an isometry. So, for each fixed $p_{k+1} \in \mathbb{Z}_{+}$, we get
\begin{align*}
(I_{\clh}-V_{k+1}^{p_{k+1}}V_{k+1}^{*p_{k+1}}) \clw_K 		
&=[\clw_K \ominus V_{k+1} \clw_K] \oplus \cdots \oplus V_{k+1}^{ p_{k+1}-1} [\clw_K \ominus V_{k+1} \clw_K]\\
&= (\clw_K \cap \clw_{k+1}) \oplus \cdots \oplus V_{k+1}^{ p_{k+1}-1}( \clw_K \cap \clw_{k+1}) \\
&= \bigoplus_{r=0}^{ p_{k+1}-1}  V_{k+1}^{r}( \clw_1 \cap \dots \cap \clw_{k+1})
\end{align*}	
as Proposition \ref{Corollary of intersection form} gives $\clw_K \ominus V_{k+1}\clw_K= \clw_K \cap \clw_{k+1}$. Now employing Lemma \ref{Main Corollary for n-tuple}, for each fixed $p_{k+1} \in \mathbb{Z}_+$, relation \eqref{eq-rev inclusion} becomes
\begin{align*}
(I_{\clh}-V_{k+1}^{p_{k+1}}V_{k+1}^{*p_{k+1}}) h & \in  \bigoplus_{p_1, \dots, p_k \in \mathbb{Z}_+} (I_{\clh}-V_{k+1}^{p_{k+1}}V_{k+1}^{*p_{k+1}}) V_1^{p_1} \cdots V_k^{p_k} (\clw_1 \cap \cdots \cap \clw_k) \\
&= \bigoplus_{p_1, \dots, p_k \in \mathbb{Z}_+} V_1^{p_1} \cdots V_k^{p_k} (I_\clh-V_{k+1}^{p_{k+1}} V_{k+1}^{*p_{k+1}}) (\clw_1 \cap \cdots \cap \clw_k) \\
&=  \bigoplus_{r=0}^{ p_{k+1}-1} \bigoplus_{p_1, \dots, p_k \in \mathbb{Z}_+} V_1^{p_1} \cdots V_k^{p_k} V_{k+1}^{r}( \clw_1 \cap \cdots \cap \clw_{k+1}).
\end{align*}	
Now $\lim\limits_{p_{k+1} \rightarrow \infty} V_{k+1}^{*p_{k+1}}|_{\clh_{s(k+1)}}=0$ 
in the strong operator topology. Letting $p_{k+1} \rightarrow \infty$, we have
\[
h  \in \bigoplus_{p_1, \dots, p_{k+1} \in \mathbb{Z}_+} V_1^{p_1} \cdots V_k^{p_k} V_{k+1}^{p_{k+1}}( \clw_1 \cap \dots \cap \clw_{k+1}).
\]
Therefore, 
\begin{equation*}
\bigcap_{i \in I_{k+1}}  \big(\bigoplus\limits_{p_i \in \mathbb{Z}_+}V_i^{p_i}\clw_{i} \big) \subseteq \bigoplus_{p_1, \dots, p_{k+1} \in \mathbb{Z}_+} V_1^{p_1} \cdots V_k^{p_k} V_{k+1}^{p_{k+1}}( \clw_1 \cap \cdots \cap \clw_{k+1}).
\end{equation*}
Hence the statement is true for $l=k+1$. This completes the proof.
\end{proof}

Now we are in a position to state and prove  our main result as follows:

\begin{thm}\label{Theorem Main Result}
Let $V=(V_1,\dots, V_n)$ be an $n$-tuple of isometries on $\clh$ such that $\clr(V_i^{m_i}V_j^{m_j})=\clr(V_j^{m_j} V_i^{m_i})$ and $\clr(V_i^{*m_i}V_j^{m_j})=\clr(V_j^{m_j} V_i^{*m_i})$ for  $m_i,m_j \in \mathbb{Z}_+$, where $1 \leq i<j \leq n$. Then there exist $2^n$ joint $V$-reducing subspaces $\{\clh_{\Lambda}: \Lambda \subseteq I_n \}$ (counting the trivial subspace $\{0\}$) such that 
\begin{align*}
\clh= \bigoplus_{\Lambda \subseteq I_n} \clh_{\Lambda},
\end{align*}
where, for  $\clh_{\Lambda} \neq \{0\}$, $V_j|_{\clh_{\Lambda}}$ is unitary if $j \in I_n \backslash \Lambda$ and $V_i|_{\clh_{\Lambda}}$ is a shift if $i \in  \Lambda$. Moreover, for each $ \Lambda \subseteq I_n $, 
\begin{align}\label{equn 4}
\clh_{\Lambda}= \bigoplus_{{\bm p} \in \mathbb{Z}_{+}^{|\Lambda|}} V_{{ \Lambda}}^{\bm p} 
\bigg( \bigcap_{{\bm q} \in \mathbb{Z}_{+}^{{n - |\Lambda|}}}V^{\bm q}_{ I_n \backslash \Lambda} \ \clw_{ \Lambda}  \bigg)
\end{align}	
and the above decomposition is unique.
\end{thm}

\begin{proof}
Let $n \geq 2$ be a fixed positive integer. Suppose $V=(V_1,\dots, V_n)$ is an $n$-tuple of isometries on $\clh$ such that $\clr(V_i^{m_i}V_j^{m_j})=\clr(V_j^{m_j} V_i^{m_i})$ and $\clr(V_i^{*m_i}V_j^{m_j})=\clr(V_j^{m_j} V_i^{*m_i})$ for  $m_i,m_j \in \mathbb{Z}_+$ where $1 \leq i<j \leq n$. Thus, by Theorem \ref{thm-Multi-isometries}, there exist $2^n$ joint $V$-reducing subspaces $\{\clh_{\Lambda}: \Lambda\subseteq I_n \} $ (including the subspace \{0\}) such that 
\begin{align*}
\clh= \bigoplus_{\Lambda \subseteq I_n}{\clh_\Lambda}
\end{align*}
and for each non-zero decomposition subspace  $\clh_{\Lambda} $, we have $V_j|_{\clh_{\Lambda}}$ is unitary if  $ j \in I_n \backslash \Lambda$ and $V_i|_{\clh_{\Lambda}}$ is a shift if  $i \in \Lambda$. Also, from the equation \eqref{equ-subspace}, we have the orthogonal decomposition spaces as
\[
\clh_{\Lambda}=\big(\bigcap_{j \in I_n \backslash \Lambda} \clh_{uj}\big) \cap \big(\bigcap_{i \in \Lambda} \clh_{si}\big).
\]
Now consider $\Lambda= \{i_1 < \dots < i_l\} \subseteq I_n$ for $ 1\leq l \leq n$, and $I_n \backslash \Lambda= \{i_{l+1} <  \dots < i_n\}$. Then Lemma \ref{Corollary for unitary} gives 
\begin{align*}
\bigcap_{j \in I_n \backslash \Lambda} \clh_{uj}
&=  \big(\bigcap_{q_{i_{l+1}} \in \mathbb{Z}_+} V_{i_{l+1}}^{q_{i_{l+1}}} \clh \big) \cap \cdots \cap \big(\bigcap_{{q_{i_n}} \in \mathbb{Z}_+} V_{i_n}^{q_{i_n}} \clh\big) 
\\
&= \bigcap_{q_{i_{l+1}, \dots , q_{i_n}} \in \mathbb{Z}_+}V_{i_{l+1}}^{q_{i_{l+1}}} \cdots V_{i_{n}}^{q_{i_n}} \clh,
\end{align*}
and Lemma \ref{Corollary for shift} yields
\begin{align*}
\bigcap_{i \in \Lambda} \clh_{si}
&=\big(\bigoplus_{p_{i_1} \in \mathbb{Z}_+} V_{i_1}^{p_{i_1}} \clw_{i_1} \big) \cap \cdots \cap \big(\bigoplus_{p_{i_l} \in \mathbb{Z}_+} V_{i_l}^{p_{i_l}} \clw_{i_l} \big) \\
&=\bigoplus_{p_{i_1}, \dots, p_{i_l} \in \mathbb{Z}_+} V_{i_1}^{p_{i_1}} \cdots  V_{i_l}^{p_{i_l}} (\clw_{i_1} \cap \dots \cap \clw_{i_l})\\
&=\bigoplus_{p_{i_1}, \dots, p_{i_l} \in \mathbb{Z}_+} V_{i_1}^{p_{i_1}} \cdots  V_{i_l}^{p_{i_l}} \clw_{\Lambda} .
\end{align*}
Therefore,
\[
\clh_{\Lambda}=\Big(\bigcap_{q_{i_{l+1}, \dots , q_{i_n}} \in \mathbb{Z}_+}V_{i_{l+1}}^{q_{i_{l+1}}} \cdots V_{i_{n}}^{q_{i_n}} \clh\Big) \cap \Big(\bigoplus_{p_{i_1}, \dots, p_{i_l} \in \mathbb{Z}_+} V_{i_1}^{p_{i_1}} \cdots  V_{i_l}^{p_{i_l}} \clw_{\Lambda} \Big).
\]
Let $h \in \Big(\bigcap\limits_{q_{i_{l+1}, \dots , q_{i_n}} \in \mathbb{Z}_+}V_{i_{l+1}}^{q_{i_{l+1}}} \cdots V_{i_{n}}^{q_{i_n}} \clh\Big)\cap \Big(\bigoplus\limits_{p_{i_1}, \dots, p_{i_l} \in \mathbb{Z}_+} V_{i_1}^{p_{i_1}} \cdots  V_{i_l}^{p_{i_l}} \clw_{\Lambda} \Big)$. Again, by Proposition \ref{Multi-Index-Reducing-Prop}, $V_j|_{\clw_{\Lambda}}$ is  isometry on $\clw_{\Lambda}$ for $j \in I_n \backslash \Lambda$. Using Lemma \ref{Main Corollary for n-tuple}, for each $q_{i_{l+1}}, \dots, q_{i_n} \in \mathbb{Z}_+$, we obtain 
\begin{align*}
V_{i_{n}}^{*q_{i_n}} \cdots	V_{i_{l+1}}^{*q_{i_{l+1}}}h 
&\in \bigoplus\limits_{p_{i_1}, \dots, p_{i_l} \in \mathbb{Z}_+}  V_{i_{n}}^{*q_{i_n}} \cdots	V_{i_{l+1}}^{*q_{i_{l+1}}}V_{i_1}^{p_{i_1}} \cdots  V_{i_l}^{p_{i_l}}  \clw_{\Lambda} \\ 
&= \bigoplus\limits_{p_{i_1}, \dots, p_{i_l} \in \mathbb{Z}_+} V_{i_1}^{p_{i_1}} \cdots  V_{i_l}^{p_{i_l}}  V_{i_{n}}^{*q_{i_n}} \cdots	V_{i_{l+1}}^{*q_{i_{l+1}}}\clw_{\Lambda}	\\
&=\bigoplus\limits_{p_{i_1}, \dots, p_{i_l} \in \mathbb{Z}_+} V_{i_1}^{p_{i_1}} \cdots  V_{i_l}^{p_{i_l}} \clw_{\Lambda}.		 
\end{align*}
Now $V_j|_{\clh_{\Lambda}}$ is unitary for $j \in I_n \backslash \Lambda$, that is, $V_{i_{l+1}}, \dots, V_{i_n}$ are unitaries on $\clh_{\Lambda}$.  Therefore, for each $q_{i_{l+1}}, \dots, q_{i_n} \in \mathbb{Z}_+$ 
\begin{align*}
h=  V_{i_{l+1}}^{q_{i_{l+1}}} \cdots V_{i_n}^{q_{i_n}} V_{i_{n}}^{*q_{i_n}} \cdots	V_{i_{l+1}}^{*q_{i_{l+1}}}h
&\in \bigoplus\limits_{p_{i_1}, \dots, p_{i_l} \in \mathbb{Z}_+} V_{i_{l+1}}^{q_{i_{l+1}}} \cdots V_{i_n}^{q_{i_n}}V_{i_1}^{p_{i_1}} \cdots  V_{i_l}^{p_{i_l}} \clw_{\Lambda}\\
&=  \bigoplus\limits_{\bm p=(p_{i_1}, \dots, p_{i_l}) \in \mathbb{Z}_+^{|\Lambda|}}    V_{\Lambda}^{\bm p}V_{i_{l+1}}^{q_{i_{l+1}}} \cdots V_{i_n}^{q_{i_n}}\clw_{\Lambda}
\end{align*}
which implies 
\[
h \in \bigcap_{q_{i_{l+1}}, \dots, q_{i_n} \in \mathbb{Z}_+} \bigoplus\limits_{\bm p \in \mathbb{Z}_+^{|\Lambda|}} V_{\Lambda}^{\bm p} V_{i_{l+1}}^{q_{i_{l+1}}} \cdots V_{i_n}^{q_{i_n}} \clw_{\Lambda}
=\bigoplus\limits_{\bm p \in \mathbb{Z}_+^{|\Lambda|}}V_{\Lambda}^{\bm p} \big( 
\bigcap_{\bm q \in \mathbb{Z}_+^{n-|\Lambda|}} V^{\bm q}_{I_n \backslash \Lambda}\clw_{\Lambda} \big).
\]
It follows that
\[
\clh_{\Lambda}=\big(\bigcap_{j \in I_n \backslash \Lambda} \clh_{uj}\big) \cap \big(\bigcap_{i \in \Lambda} \clh_{si}\big) \subseteq \bigoplus\limits_{\bm p \in \mathbb{Z}_+^{|\Lambda|}}V_{\Lambda}^{\bm p} \big( \bigcap_{\bm q \in \mathbb{Z}_+^{n-|\Lambda|}} V^{\bm q}_{I_n \backslash \Lambda}\clw_{\Lambda} \big).
\]

To prove the reverse inclusion, we know that $\clw_{\Lambda} \subseteq \clw_{i}$ for all $ i \in \Lambda$. Now, using Lemma \ref{Main Corollary for n-tuple}, we obtain
\[
\bigoplus\limits_{p_{i_1}, \dots, p_{i_l} \in \mathbb{Z}_+} V_{i_1}^{p_{i_1}} \cdots  V_{i_l}^{p_{i_l}} \clw_{\Lambda}\subseteq \bigoplus\limits_{p_{i_1}, \dots, p_{i_l} \in \mathbb{Z}_+} V_{i_1}^{p_{i_1}} \cdots  V_{i_l}^{p_{i_l}} \clw_{i} \subseteq \bigoplus\limits_{p_{i} \in \mathbb{Z}_+}V_{i}^{p_{i}}\clw_{i}.
\] 
Also
$ \bigoplus\limits_{p_{i_1}, \dots, p_{i_l} \in \mathbb{Z}_+} V_{i_1}^{p_{i_1}} \cdots  V_{i_l}^{p_{i_l}} \clw_{\Lambda}\subseteq  \clh $. Therefore,
\[
\bigoplus\limits_{\bm p \in \mathbb{Z}_+^{|\Lambda|}}V_{\Lambda}^{\bm p} \big( 
\bigcap_{\bm q \in \mathbb{Z}_+^{n - |\Lambda|}} V^{\bm q}_{I_n \backslash \Lambda}\clw_{\Lambda} \big)  \subseteq \bigcap_{j \in I_n \backslash \Lambda}  \big(\bigcap_{q_{j} \in \mathbb{Z}_+} V_{j}^{q_{j}}\clh \big) \cap 
\bigcap_{i \in \Lambda }   \big(\bigoplus\limits_{p_{i} \in \mathbb{Z}_+}V_{i}^{p_{i}}\clw_{i} \big) =\clh_{\Lambda}.
\]
Hence, from both the inclusions,
\[
\clh_{\Lambda}= \bigoplus\limits_{\bm p \in \mathbb{Z}_+^{|\Lambda|}}V_{\Lambda}^{\bm p} \big( \bigcap_{\bm q \in \mathbb{Z}_+^{n - |\Lambda|}} V^{\bm q}_{I_n \backslash \Lambda}\clw_{\Lambda} \big).
\]
This finishes the proof.
\end{proof}

There are few remarks in order:

\begin{rem}
\begin{enumerate}
\item One can prove the above theorem by classical mathematical induction method. However, our approach is here more analytical, new and applicable in general setting.

\item Let $(V_1, \ldots, V_n)$ be an $n$-tuple of $\clu_n$-twisted isometries on 
$\clh$ with respect to a twist $\{U_{ij}\}_{i<j}$, that is,
\begin{center}
$V_iV_j=U_{ij}V_jV_i, \hspace{0.5 cm}V_i^*V_j= U^*_{ij}V_jV_i^*,  \
\ V_kU_{ij} =U_{ij}V_k ~~\mbox{for all} \ i,j,k=1, \dots, n \ \text{and}  \ i \neq j$
\end{center}
where $\{U_{ij}\}_{1 \leq i < j \leq n}$ is $\binom{n}{2}$ commuting unitaries on 
$\clh$ such that $U_{ji}:=U^*_{ij}$. Note that a complete structure for $\clu_n$-twisted isometries has been established in \cite{MM-Wold-type} (see also \cite{RSS-TWISTED ISOMETRIES}, \cite{SARKAR-WOLD}). It is noteworthy to mention that if $(V_1,\dots,  V_n)$ is an $n$-tuple of $\clu_n$-twisted isometry on $\clh$,  then clearly $\clr(V_i^{m_i} V_j^{m_j})= \clr(V_j^{m_j}V_i^{m_i})$ and  
$\clr(V_i^{*m_i} V_j^{m_j})= \clr(V_j^{m_j}V_i^{*m_i})$ for $1 \leq i <j \leq n$, where $m_{i}, m_j \in \mathbb{Z}_+$. Consequently, Theorem  \ref{Theorem Main Result} gives {\textit{Wold decomposition}} for larger class than $\clu_n$-twisted isometries, and the orthogonal decomposition spaces coincide with  
orthogonal decomposition spaces for $n$-tuple of doubly commuting isometries. 
\end{enumerate}
\end{rem}

We also obtain the structure of $n$-tuple of isometries in the same line as in \cite{BKS-CANONICAL} in the case of finite dimensional wandering subspace.

\begin{cor}
Let $V=(V_1, \dots, V_n)$ be an n-tuple of isometries on $\clh$ such that $\clr(V_i^{m_i} V_j^{m_j})=\clr(V_j^{m_j} V_i^{m_i})$ for $m_i, m_j \in \mathbb{Z}_+$ and dim $\cln (V_i^*) < \infty$ for all $i, j =1,\dots,n$. Then there is a unique decomposition having $2^n$-joint $V$-reducing subspaces $\{\clh_{\Lambda}: \Lambda \subseteq I_n\}$ such that 
\[
\clh= \bigoplus_{\Lambda \subseteq I_n}\clh_{\Lambda},
\] 
where for $\clh_{\Lambda} \neq \{0\}$, $V_j|_{\clh_{\Lambda}}$ is unitary if $j \in I_n \backslash \Lambda$, and $V_i|_{\clh_{\Lambda}}$ is a shift if $i \in \Lambda$. 
\end{cor}

\section{Analytic Models}

This section deals with the analytic model for $n$-tuples of {\textit{isometries with equal range}}. Recall that the pairs $(M_z, \Delta[{\bm{r}}])$ on $H^2(\mathbb{D})$, and $(M_{z_1},M_{z_2}(\Delta[\bm r] \otimes I_{H^2(\mathbb{D}} ))$ on $H^2(\mathbb{D}^2)$ given in Example \ref{Example3-proper}, are typical examples of pairs of {\textit{isometries with equal range}}, and these are actually ``building blocks" in the analytic models.

Before proceeding to the main results, we recall some standard notations and definitions:
For $l \geq 1$, let $\D^l = \D \times \dots \times \D$ be the open unit polydisc in 
$\mathbb{C}^l$. We use the notation $\z$ for the $l$-tuple $(z_1, \ldots, z_l)$ in $\mathbb{C}^l$. Also, for any multi-index $\bm{k} = (k_1, \ldots, k_l) \in \mathbb{Z}_+^l$ and $\z \in \mathbb{C}^l$, we write $\z^{\bm{k}} = z_1^{k_1} \cdots z_l^{k_l}$.
Two tuples $V=(V_1, \ldots, V_n)$ on a Hilbert space $\clh$ and $T=(T_1, \ldots, T_n)$
on a Hilbert space $\clk$ are said to be unitarily equivalent if there exists a unitary $U: \clh \rightarrow \clk$ such that $UV_i = T_iU$ for all $i=1, \ldots, n$.

Let $\cle$ be a separable complex Hilbert space. Then $\cle$-valued Hardy space over the unit polydisc 
$\D^l$, denoted by $H^2_{\cle}(\D^l)$, is the Hilbert space of all holomorphic functions $f(\bm z) = \sum_{ \bm{k} \in \mathbb{Z}^l_+ } a_{\bm k} \bm{z}^{\bm k}$, where $a_{\bm k} \in \cle$, on $\mathbb{D}^l$ such that
\[
\|f\|_{H^2_{\cle}(\D^l)} := \left( \sum_{ \bm{k} \in \mathbb{Z}^l_+ } \| a_{\bm k} \|_{\cle}^2 \right)^{\frac{1}{2}}< \infty.
\]
We shall frequently identify the $\cle$-valued Hardy space $H^2_{\cle}(\D^l)$ 
with $H^2(\D^l)\otimes \cle$ by the canonical unitary map $\Pi :
H^2_{\cle}(\D^l) \raro H^2(\D^l) \otimes \cle$ defined by
\[
\Pi (\z^{\bm{k}} \eta) = \z^{\bm{k}} \otimes \eta \quad \quad (\bm{k}
\in \mathbb{Z}^l_+, \eta \in \cle).
\]
It is again easy to observe that the scalar valued Hardy space $H^2(\D^l)$, $l>1$ 
over $\D^l$ can be identified via the natural unitary map
\[ 
z_1^{k_1} \cdots z_{l}^{k_{l}} \mapsto z^{k_1}\otimes \cdots \otimes z^{k_l}
\]
with the $l$-fold Hilbert space tensor product $H^2(\D) \otimes \cdots \otimes H^2(\D)$ of the Hardy space $H^2(\D)$ over the unit disc $\D$. Let $(M_{z_1}, \ldots, M_{z_l})$ denote the $l$-tuple of multiplication operators by the co-ordinate functions $\{z_j\}_{j=1}^l$ on $H^2_{\cle}(\D^l)$, that is,
\[
(M_{z_j} f)(\w) = w_j f(\w),
\]
for all $f \in H^2_{\cle}(\D^l)$, $\w \in \D^l$ and $j = 1, \ldots, l$. It is also easy to check that $(M_{z_1}, \ldots, M_{z_l})$ is an $l$-tuple of \textit{doubly commuting} shifts on $H^2_{\cle}(\D^l)$. Clearly, the shift $M_{z_j}$ on $H^2_{\cle}(\D^l)$ can be identified with $M_{z_j} \otimes I_{\cle}$ on $H^2(\D^l) \otimes \cle$ for $j = 1, \ldots, l$. This canonical identification will be used throughout the rest of the paper.

Returning back to the analytic representation of a single isometry, let $V$ be an isometry on $\clh$. Then the {{\textit{Wold decomposition}} for the isometry $V$ 
gives $\clh=\clh_{\emptyset} \oplus \clh_{\{1\}}$,
where $\clh_{\emptyset}= \cap_{m=0}^{\infty} V^m\clh$, $\clh_{\{1\}}= \oplus_{m=0}^{\infty} V^m \clw$ and $\clw=\clh \ominus V\clh$. Now define a map 
$\tau_V: \clh_{\emptyset} \oplus \clh_{\{1\}} \rightarrow \clh_{\emptyset} \oplus H^2_{\clw}(\mathbb{D}) $  by
\[
\tau_V (h_u \oplus V^m \eta) = h_u \oplus z^m \eta \quad (m \in \mathbb{Z}_+, h_u \in \clh_{\emptyset}, \eta \in \clw).
\]
Clearly, $\tau_V$ is unitary, and also 
\[
\tau_V \begin{bmatrix}
	V|_{\clh_{\emptyset}} &0 \\
	0   & V|_{\clh_{\{1\}}}
\end{bmatrix}
= \begin{bmatrix}
	V|_{\clh_{\emptyset}} & 0 \\
	0   & M_z    
\end{bmatrix} \tau_V,
\]
that means, $V=V|_{\clh_{\emptyset}} \oplus V|_{\clh_{\{1\}}}$ on $\clh$  is unitarily equivalent to $V|_{\clh_{\emptyset}} \oplus  M_z $  on $ \clh_{\emptyset} \oplus H^2_{\clw}(\mathbb{D})$. It follows that the shift part of the isometry $V$ yields an analytic representation in terms of multiplication operator $M_z$ on the $\clw$-valued Hardy space $H^2_{\clw}(\mathbb{D})$ over the unit disc.

Inspiring from the above analytic structure of a single isometry as well as analytic models for certain tuples of isometries described in \cite{JP-NON COMMUTING} \& \cite{RSS-TWISTED ISOMETRIES}, we intend to find the analytic models for tuples of {\textit{isometries with equal range}}. Before moving on to the main results, we introduce some fundamental concepts from \cite{JP-NON COMMUTING} \& \cite{RSS-TWISTED ISOMETRIES}.

\begin{defn}
Let $j \in \{ 1, \ldots, l \}$. Let $U= (U^{(m)})$ be a sequence of unitaries on 
$\cle$. The $j$-th diagonal operator with symbol $U$, denoted by $\Delta_{j}[U]$, is a unitary operator on $H^2_{\cle}(\mathbb{D}^l)$ defined as
\[
\Delta_{j}[U ] (\bm{z}^{\bm k}\eta)=\bm{z}^{\bm k} (U^{(k_{j})}\eta) 
\quad \quad (\bm{k} \in \mathbb{Z}^l_+, \eta \in \cle). 
\]
\end{defn}

\begin{defn}\label{Lambda-wandering data} 
Let $V=(V_1, \ldots, V_n)$ be an $n$-tuple of isometries with equal range on $\clh$. 
Suppose $\Lambda=\{i_1< \dots< i_l\} \subseteq I_n$ for $1 \leq l \leq n$, and $I_n \backslash  \Lambda=\{i_{l+1}< \dots< i_n \}$. The $\Lambda$-wandering subspace of the tuple $V$ is defined as
\[
\cls_{\Lambda}(V)= \bigcap_{\bm{q} \in \mathbb{Z}_+^{n- |\Lambda|}} V_{I_n \backslash \Lambda}^{\bm q} \clw_{\Lambda},
\]
and the $\Lambda$-wandering data of $V$ as the $(n-l+1)$-tuple   
\[
wd_{\Lambda}(V)= (I|_{\cls_{\Lambda}(V)},V_{i_{l+1}}|_{\cls_{\Lambda}(V)},\dots, V_{i_n}|_{\cls_{\Lambda}(V)})
\]
on $\cls_{\Lambda}(V)$.
\end{defn}

If the tuple $V$ is obvious from the context, we will simply write $\cls_{\Lambda}(V)$ as $\cls_{\Lambda}$. The wandering data $wd_{\Lambda}(V)$ defined above is guaranteed to be a well-defined $(n-l+1)$-tuple on $\cls_{\Lambda}$ by the following lemma:

\begin{lem}\label{Lemma- Reduce A wandering subspace }
Let $V=(V_1, \ldots, V_n)$ be an $n$-tuple of isometries with equal range on $\clh$. If $\Lambda \subseteq I_n$, then $\cls_{\Lambda}$ reduces $V_j$ for $j \in  I_n \backslash \Lambda$. Moreover, $V_j|_{\cls_{\Lambda}}$ is unitary for $j \in I_n \backslash \Lambda$ and $V_i^*|_{\cls_{\Lambda}}=0$ for $i \in \Lambda$.  
\end{lem}

\begin{proof}
Suppose $\Lambda=\{i_1<\dots< i_l\}$ for $ 1 \leq l < n$ and $I_n \backslash \Lambda =\{i_{l+1}< \dots< i_n\}$. Then $\clw_{\Lambda}$ is $V_j$-reducing subspace for $j \in I_n \backslash  \Lambda$ by Proposition \ref{Multi-Index-Reducing-Prop}. Now Lemma \ref{Reducing-subspace-with-equal-range} and Lemma \ref{Lemma- Main result} yield
\[
V_{i_t}^{q_{i_t}}V_{i_s}^{q_{i_s}}\clw_\Lambda =V_{i_s}^{q_{i_s}}V_{i_t}^{q_{i_t}}\clw_\Lambda \quad \text{and} \quad 
V_{i_t}^{*q_{i_t}}V_{i_s}^{q_{i_s}}\clw_\Lambda =V_{i_s}^{q_{i_s}}V_{i_t}^{*q_{i_t}}\clw_\Lambda
\]
where $i_t,i_s \in I_n \backslash \Lambda$ and $q_{i_t},q_{i_s} \in \mathbb{Z}_+$. Applying Lemma \ref{Important Lemma for n-tuples}, we have 
\begin{align*}
&V_{i_{l+1}}^{q_{i_{l+1}}} V_{i_{l+2}}^{q_{i_{l+2}}}\cdots V_{i_n}^{q_{i_n}}\clw_{\Lambda}= V_{i_{l+2}}^{q_{i_{l+2}}} \cdots V_{i_n}^{q_{i_n}} V_{i_{l+1}}^{q_{i_{l+1}}}\clw_{\Lambda},	 \\
\quad 
\text{and} \quad  & V_{i_{l+1}}^{*q_{i_{l+1}}} V_{i_{l+2}}^{q_{i_{l+2}}}\cdots V_{i_n}^{q_{i_n}}\clw_{\Lambda}= V_{i_{l+2}}^{q_{i_{l+2}}}\cdots V_{i_n}^{q_{i_n}} V_{i_{l+1}}^{*q_{i_{l+1}}}\clw_{\Lambda}
\end{align*}
where $q_{i_{l+1}}, \dots, q_{i_n} \in \mathbb{Z}_+$.

Let $j=i_r\in I_n \backslash \Lambda$ for $ r \in \{l+1, \dots, n\}$. Then using the above relations we get
\begin{align*}
V_{i_r}  (\bigcap\limits_{\bm q \in \mathbb{Z}_+^{n-|\Lambda|}}V_{I_n \backslash \Lambda}^{\bm q}\clw_{\Lambda})
&= \bigcap\limits_{\bm q \in \mathbb{Z}_+^{n-|\Lambda|}}V_{i_r} (V_{i_{l+1}}^{q_{i_{l+1}}}  \cdots V_{i_{r-1}}^{q_{i_{r-1}}}V_{i_{r}}^{q_{i_{r}}}V_{i_{r+1}}^{q_{i_{r+1}}}\cdots V_{i_n}^{q_{i_n}}\clw_{\Lambda})\\
&= \bigcap\limits_{\bm q \in \mathbb{Z}_+^{n-|\Lambda|}}V_{i_r}^{q_{i_r}+1} V_{i_{r+1}}^{q_{i_{r+1}}}\cdots V_{i_n}^{q_{i_n}} V_{i_{l+1}}^{q_{i_{l+1}}} \cdots V_{i_{r-1}}^{q_{i_{r-1}}}\clw_{\Lambda}\\
&= \bigcap\limits_{\bm q \in \mathbb{Z}_+^{n-|\Lambda|}}V_{i_r}^{q_{i_r}} V_{i_{r+1}}^{q_{i_{r+1}}}\cdots V_{i_n}^{q_{i_n}} V_{i_{l+1}}^{q_{i_{l+1}}} \cdots V_{i_{r-1}}^{q_{i_{r-1}}}\clw_{\Lambda}\\
&= \bigcap\limits_{\bm q \in \mathbb{Z}_+^{n-|\Lambda|}} V_{i_{l+1}}^{q_{i_{l+1}}}  \cdots V_{i_{r}}^{q_{i_{r}}}\cdots V_{i_n}^{q_{i_n}}\clw_{\Lambda} \\
&= \bigcap\limits_{\bm q \in \mathbb{Z}_+^{n-|\Lambda|}} V_{I_n \backslash \Lambda}^{\bm q}\clw_{\Lambda}.
\end{align*}
If $\Lambda =\emptyset$, then $\clw_{\emptyset}=\clh$. Now Theorem \ref{Theorem Main Result} asserts that $\cls_{\Lambda}=\clh_{\emptyset}$ reduces $V_j$ for all $j \in I_n$. On the other hand, if $\Lambda =I_n$, the statement holds trivially.
Therefore, 
\[	
\cls_\Lambda(V)= \bigcap\limits_{\bm q \in \mathbb{Z}_+^{n-|\Lambda|}}
V_{I_n \backslash \Lambda}^{\bm q}\clw_{\Lambda} 
\]
reduces each $V_j$ for $j \in I_n \backslash \Lambda$. 	
	
Finally, $V_j(\cls_{\Lambda})=\cls_{\Lambda}$ implies that $V_j|_{\cls_{\Lambda}}$ is unitary for $j \in I_n \backslash \Lambda$. Clearly, $\cls_{\Lambda} \subseteq \clw_{\Lambda}$ for all $\Lambda \subseteq I_n$, and hence $V_i^*|_{\cls_{\Lambda}}=0$ for all $i \in \Lambda$.
\end{proof}

The next result on the $\Lambda$-wandering subspace is significant in the continuation. 

\begin{lem}\label{Lemma-Equal on A-wandering subspace}
Let $(V_1,\dots, V_n)$ be an $n$-tuple of isometries with equal range on $\clh$. 
If $\Lambda \subseteq I_n$, then
\begin{enumerate}
\item[(a)] $V_{i_j}V_{\Lambda}^{\bm p} \cls_{\Lambda}= V_{\Lambda}^{\bm p+e_j}  \cls_{\Lambda}$ for $i_j \in  \Lambda$.

\item[(b)] $V_{i_k} V_{\Lambda}^{\bm p} \cls_{\Lambda}= V_{\Lambda}^{\bm p} V_{i_k} \cls_{\Lambda}$ for $i_k \in I_n \backslash \Lambda$.
\end{enumerate} 	
\end{lem}

\begin{proof}
Let $\Lambda=\{i_1< \dots< i_l\}$ for $1 \leq l < n$ and $I_n \backslash \Lambda=\{i_{l+1}< \dots < i_n\}$. Recall the proof of part (2) of Lemma \ref{Main Corollary for n-tuple} that
 \[
\clw_{\Lambda}=\clh - \sum V_{i_p} \clh + \sum V_{i_p} V_{i_q}\clh-\dots +(-1)^lV_{i_1}V_{i_2}\cdots V_{i_l} \clh.
 \]   
$(a)$ If $i_j \in \Lambda$, then, using Lemma \ref{Important Lemma for n-tuples} repeatedly on the second line below, we have 
\begin{align*}
V_{i_j}V_{\Lambda}^{\bm p} \cls_{\Lambda}
&=\bigcap\limits_{\bm q \in \mathbb{Z}_+^{n-|\Lambda|}} V_{i_j}V_{\Lambda}^{\bm p} V_{I_n \backslash \Lambda}^{\bm q}\clw_{\Lambda}\\
&=\bigcap\limits_{\bm q \in \mathbb{Z}_+^{n-|\Lambda|}}V_{i_j} V_{i_1}^{p_{i_1}}\cdots V_{i_j}^{p_{i_j}}\cdots V_{i_l}^{p_{i_l}} V_{I_n \backslash \Lambda}^{\bm q}[\clh - \sum V_{i_p} \clh +\dots +(-1)^lV_{i_1}V_{i_2}\cdots V_{i_l} \clh]  \\
&=\bigcap\limits_{\bm q \in \mathbb{Z}_+^{n-|\Lambda|}} V_{i_1}^{p_{i_1}}\cdots V_{i_j}^{p_{i_j}+1}\cdots V_{i_l}^{p_{i_l}} V_{I_n \backslash \Lambda}^{\bm q} \clw_{\Lambda} \\
&=\bigcap\limits_{\bm q \in \mathbb{Z}_+^{n-|\Lambda|}}V_{\Lambda}^{\bm p+e_j} V_{I_n \backslash \Lambda}^{\bm q}\clw_{\Lambda}\\
&=V_{\Lambda}^{\bm p+e_j}\cls_{\Lambda}.
\end{align*}

$(b)$ If $i_k \in  I_n \backslash \Lambda$, then, again using Lemma \ref{Important Lemma for n-tuples}, we obtain
\[
V_{i_k} V_{\Lambda}^{\bm p} V_{I_n \backslash \Lambda}^{\bm q}\clw_{\Lambda}
=  V_{\Lambda}^{\bm p} V_{i_k}V_{I_n \backslash \Lambda}^{\bm q}\clw_{\Lambda}.
\]
Therefore, 
\begin{align*}
 V_{i_k}V_{\Lambda}^{\bm p}\cls_{\Lambda}=  \bigcap\limits_{\bm q \in \mathbb{Z}_+^{n-|\Lambda|}}
 V_{i_k} V_{\Lambda}^{\bm p} V_{I_n \backslash \Lambda}^{\bm q}\clw_{\Lambda}
 =\bigcap\limits_{\bm q \in \mathbb{Z}_+^{n-|\Lambda|}} 
 V_{\Lambda}^{\bm p}V_{i_k} V_{I_n \backslash \Lambda}^{\bm q} \clw_{\Lambda} 
= V_{\Lambda}^{\bm p}V_{i_k}  \cls_{\Lambda}.
\end{align*}
\end{proof}

Next, we note the observation that follows from the above lemma and Lemma \ref{Lemma-Unitary-equal range}.

\begin{prop}\label{Remark-for unitary}
Let $(V_1,\ldots, V_n)$ be an $n$-tuple of {\textit{isometries with equal range}} on 
$\clh$. If $\Lambda=  \{i_1< \dots< i_l\} \subseteq I_n$, then
\begin{enumerate}
\item[(1)] for each $i_j \in  \Lambda$, there exists a unique unitary operator $U_{i_j} \in \clb(\clh)$ such that 
\[
V_{i_j}V_{\Lambda}^{\bm p}= V_{\Lambda}^{\bm p+e_j}U_{i_j},			
\]  
where $U_{i_j}=U_{i_ji_1}^{(p_{i_1})} U_{i_ji_2}^{(p_{i_2})}\cdots U_{i_ji_{j-1}}^{(p_{i_{j-1}})}$, and $U_{i_ji_1}^{(p_{i_1})},\dots, U_{i_ji_{j-1}}^{(p_{i_{j-1}})}$ are unitaries on $\clh$.
		 
\item[(2)] for each $i_k \in  I_n \backslash \Lambda$, there exists a unique unitary operator $U_{i_k} \in \clb(\clh)$ such that   
\[
V_{i_k} V_{\Lambda}^{\bm p}= V_{\Lambda}^{\bm p} V_{i_k}U_{i_k}, 
\]  
where $U_{i_k}=U_{i_ki_1}^{(p_{i_1})}U_{i_ki_2}^{(p_{i_2})} \cdots U_{i_ki_{l}}^{(p_{i_l})}$, and $U_{i_ki_1}^{(p_{i_1})},\dots, U_{i_ki_{l}}^{(p_{i_l})}$ are unitaries on $\clh$.
\end{enumerate}
\end{prop}

\begin{proof}

$(1)$ Suppose $i_j \in \Lambda=  \{i_1< \dots< i_l\} \subseteq I_n$. 
If $i_j=i_1$, then $U_{i_j}=I_{\clh}$.  

For $i_j \neq i_1$, using Lemma \ref{Important Lemma for n-tuples}, we have
\[
V_{i_j}  V_{i_1}^{p_{i_1}}V_{i_2}^{p_{i_2}}\cdots V_{i_l}^{p_{i_l}}\clh=V_{i_1}^{p_{i_1}}V_{i_j} V_{i_2}^{p_{i_2}}\cdots V_{i_l}^{p_{i_l}}\clh.
\]
Now, by Lemma \ref{Lemma-Unitary-equal range}, there exists a unique unitary 
$U_{i_ji_1}^{(p_{i_1})} \in \clb(\clh)$ such that 
\[
V_{i_j} V_{\Lambda}^{\bm p}= V_{i_1}^{p_{i_1}}V_{i_j} V_{i_2}^{p_{i_2}}\cdots V_{i_l}^{p_{i_l}} 	U_{i_ji_1}^{(p_{i_1})}.
\]
Let $i_j\neq i_2$. Then Lemma \ref{Important Lemma for n-tuples} yields
\begin{align*}
V_{i_j} V_{\Lambda}^{\bm p}\clh
&= V_{i_1}^{p_{i_1}}V_{i_j} V_{i_2}^{p_{i_2}}\cdots V_{i_l}^{p_{i_l}} U_{i_ji_1}^{(p_{i_1})}\clh \\
&= V_{i_1}^{p_{i_1}} V_{i_2}^{p_{i_2}}V_{i_j} V_{i_3}^{p_{i_3}} \cdots V_{i_l}^{p_{i_l}} U_{i_ji_1}^{(p_{i_1})}\clh.
\end{align*}
Again Lemma \ref{Lemma-Unitary-equal range} gives a unique unitary $U_{i_ji_2}^{(p_{i_2})} \in \clb(\clh)$ such that 
\[
V_{i_j} V_{\Lambda}^{\bm p} = V_{i_1}^{p_{i_1}} V_{i_2}^{p_{i_2}}V_{i_j} V_{i_3}^{p_{i_3}} \cdots V_{i_l}^{p_{i_l}} U_{i_ji_1}^{(p_{i_1})}U_{i_ji_2}^{(p_{i_2})}.
\]
Applying Lemma \ref{Lemma-Unitary-equal range} repeatedly, we obtain $(j-1)$  unitaries $U_{i_ji_1}^{(p_{i_1})},U_{i_ji_2}^{(p_{i_2})} \dots, U_{i_ji_{j-1}}^{(p_{i_{j-1}})}$ on $\clh$ such that
\begin{align*}
V_{i_j} V_{\Lambda}^{\bm p}
&=V_{i_1}^{p_{i_1}} V_{i_2}^{p_{i_2}} \cdots V_{i_j}^{p_{i_j}+1} \cdots V_{i_l}^{p_{i_l}} U_{i_ji_1}^{(p_{i_1})}U_{i_ji_2}^{(p_{i_2})} \cdots U_{i_ji_{j-1}}^{(p_{i_{j-1}})}\\
&= V_{\Lambda}^{\bm p +e_j}U_{i_ji_1}^{(p_{i_1})}U_{i_ji_2}^{(p_{i_2})} \cdots U_{i_ji_{j-1}}^{(p_{i_{j-1}})}\\
&=V_{\Lambda}^{\bm p +e_j}U_{i_j},
\end{align*}
where $U_{i_j}=U_{i_ji_1}^{(p_{i_1})}U_{i_ji_2}^{(p_{i_2})} \cdots U_{i_ji_{j-1}}^{(p_{i_{j-1}})}$ is a unique unitary operator $\clh$.

$(2)$ 
If $\Lambda=\emptyset$, then $U_{i_k}=I_{\clh}$ for all $i_k \in I_n \backslash \Lambda$. Suppose $\Lambda \neq \emptyset$. Then Lemma \ref{Important Lemma for n-tuples} yields
\[
V_{i_k}  V_{i_1}^{p_{i_1}}V_{i_2}^{p_{i_2}}\cdots V_{i_l}^{p_{i_l}}\clh=V_{i_1}^{p_{i_1}}V_{i_k} V_{i_2}^{p_{i_2}}\cdots V_{i_l}^{p_{i_l}}\clh.
\]
Now Lemma \ref{Lemma-Unitary-equal range} gives a unique unitary ${U}_{i_ki_1}^{(p_{i_1})} \in \clb(\clh)$ such that 
\[
 V_{i_k} V_{\Lambda}^{\bm p}= V_{i_1}^{p_{i_1}}V_{i_k} V_{i_2}^{p_{i_2}}\cdots V_{i_l}^{p_{i_l}} 	{U}_{i_ki_1}^{(p_{i_1})}.
\]
Repeating this process $l$-times, we obtain $l$ unitaries  ${U}_{i_ki_1}^{(p_{i_1})}, \dots, {U}_{i_ki_{l}}^{(p_{i_{l}})} \in \clb(\clh)$ such that
\begin{align*}
V_{i_k} V_{\Lambda}^{\bm p}
&=V_{i_1}^{p_{i_1}}V_{i_2}^{p_{i_2}}  \cdots V_{i_l}^{p_{i_l}} V_{i_k} {U}_{i_ki_1}^{(p_{i_1})} {U}_{i_ki_2}^{(p_{i_2})} \cdots {U}_{i_ki_{l}}^{(p_{i_{l}})} \\
&=V_{\Lambda}^{\bm p } V_{i_k}U_{i_k},
\end{align*}	
where $U_{i_k}= U_{i_ki_1}^{(p_{i_1})} \cdots U_{i_ki_{l}}^{(p_{i_{l}})}$ is a unique unitary operator on $\clh$.  
\end{proof}

From the above proposition and Lemma \ref{Lemma-Equal on A-wandering subspace}, we record the following observations which will be useful in finding the analytic models of tuples.

\begin{rem}\label{Remark-Reducingspace-unitary}
Let $(V_1,\dots, V_n)$ be an $n$-tuple of {\textit{isometries with equal range}} on 
$\clh$. Suppose $\Lambda=\{i_1<\dots<i_l\} \subseteq I_n$. Then
\begin{enumerate}
\item[(i)] $V_{i_j}V_{\Lambda}^{\bm p} \cls_{\Lambda}= V_{\Lambda}^{\bm p+e_j}  \cls_{\Lambda}$ for $i_j \in  \Lambda$, and $V_{i_k}V_{\Lambda}^{\bm p} \cls_{\Lambda}= V_{\Lambda}^{\bm p}V_{i_k}  \cls_{\Lambda}$ for $i_k \in I_n \backslash \Lambda$ infer that $\cls_{\Lambda}$ reduce the unitaries $U_{i_ji_1}^{(p_{i_1})},\ldots, U_{i_ji_{j-1}}^{(p_{i_{j-1}})}$, and $U_{i_ki_1}^{(p_{i_1})}, \ldots, U_{i_ki_{l}}^{(p_{i_{l}})}$ obtained in Proposition \ref{Remark-for unitary}. Thus $U_{i_j}=U_{i_ji_1}^{(p_{i_1})}\cdots U_{i_ji_{j-1}}^{(p_{i_{j-1}})}$ and $U_{i_k}=U_{i_ki_1}^{(p_{i_1})} \cdots U_{i_ki_{l}}^{(p_{i_{l}})}$ are unitaries on $\cls_{\Lambda}$.

\item[(ii)] Lemma \ref{Reducing-subspace-with-equal-range} yields $V_{i_j}V_{\Lambda}^{\bm p} \clh_{\Lambda}= V_{\Lambda}^{\bm p+e_j}\clh_{\Lambda}$  for $i_j \in \Lambda$ and $V_{i_k}V_{\Lambda}^{\bm p} \clh_{\Lambda}= V_{\Lambda}^{\bm p}V_{i_k}\clh_{\Lambda}$  for $i_k \in I_n \backslash \Lambda$. Thus, $\clh_{\Lambda}$ reduce the unitaries $U_{i_ji_1}^{(p_{i_1})},\dots, U_{i_ji_{j-1}}^{(p_{i_{j-1}})}$ and $U_{i_ki_1}^{(p_{i_1})}, \dots, U_{i_ki_{l}}^{(p_{i_{l}})}$, which are obtained in Proposition \ref{Remark-for unitary}, and hence $U_{i_j}=U_{i_ji_1}^{(p_{i_1})}\cdots U_{i_ji_{j-1}}^{(p_{i_{j-1}})}$ and $U_{i_k}=U_{i_ki_1}^{(p_{i_1})} \cdots U_{i_ki_{l}}^{(p_{i_{l}})}$ are unitaries on $\clh_{\Lambda}$. 
\end{enumerate}	
\end{rem}

Recall that the non-zero shift part (if any) of a single isometry due to {\it{Wold decomposition}} yields the analytic structure of the isometry. Now Theorem \ref{Theorem Main Result} says that an $n$-tuple of isometries $V=(V_1, \ldots, V_n)$ with equal range on $\clh$ admits {\it{Wold decomposition}}. More precisely, there exist $2^n$ joint $V$-reducing subspaces $\clh_{\Lambda}$ such that $\clh=\bigoplus\limits_{\Lambda \subseteq I_n} \clh_{\Lambda}$, where
$\clh_{\Lambda}= \bigoplus\limits_{\bm p \in \mathbb{Z}_+^{|\Lambda|}} V_{\Lambda}^{\bf p}\cls_{\Lambda}$; and for $\clh_{\Lambda} \neq \{0\}$, $V_j|_{\clh_{\Lambda}}$ is unitary if $j \in I_n \backslash \Lambda$ and $V_i|_{\clh_{\Lambda}}$ is a shift if $i \in  \Lambda$. Thus, it is important to separate the shift parts (if any) from the tuple of isometries to describe the analytic structure of $n$-tuples of {\it{isometries with equal range}}. To do that we firstly find out the analytical construction of the shift part $V|_{\clh_\Lambda}$ for $\Lambda \subseteq I_n$. Since $V|_{\clh_{\emptyset}}$ is unitary, we will discuss when $\Lambda \neq \emptyset$.

Suppose that $\clh_{\Lambda} \neq \{0\}$, where $\Lambda=\{i_1<\dots<i_l\} \subseteq I_n$ for some $1 \leq l \leq n$ and $|\Lambda|=l$. Now for each orthogonal decomposition space $\clh_{\Lambda} = \bigoplus\limits_{\bm p \in \mathbb{Z}_+^{l}}V_{\Lambda}^{\bm p} \cls_{\Lambda}$, define the canonical unitary $\tau_{\Lambda}:\clh_{\Lambda} \rightarrow H^2_{\cls_{\Lambda}}(\mathbb{D}^l)$ as
\begin{align}\label{Define unitary operator Pi }
\tau_{\Lambda}(V_{\Lambda}^{\bm p} \xi)= {\bm z}^{\bm p} \xi \quad (\ \bm p=(p_1,\dots,p_l) \in \mathbb{Z}_+^l, \ \xi \in \cls_{\Lambda}).
\end{align}
Therefore, 
 \[
(\tau_{\Lambda}V_{i_1} \tau_\Lambda^*)(\bm z^{\bm p} \xi)=\tau_{\Lambda}(V_{i_1}V_{\Lambda}^{\bm p} \xi)= \tau_{\Lambda}(V_{i_{1}}^{p_1 +1} V_{i_2}^{p_{2}}\cdots V_{i_l}^{p_{l}}\xi)= z_1({\bm z}^{\bm p} \xi)=M_{z_1}({\bm z}^{\bm p}\xi),
 \]
that means $\tau_\Lambda V_{i_1} =M_{z_1}\tau_\Lambda$. Now if $i_j \in \Lambda$ for  $j \in \{2,3,\dots,l\}$, then, by Proposition \ref{Remark-for unitary}, we obtain
\begin{align*}
(\tau_{\Lambda}V_{i_j} \tau_\Lambda^*)({\bm z}^{\bm p} \xi)
	&=\tau_{\Lambda}(V_{i_j}V_{\Lambda}^{\bm p} \xi)\\
	&=\tau_{\Lambda}(V_{\Lambda}^{\bm p +e_j} U_{i_j}\xi)\\
	&=\tau_{\Lambda}(V_{\Lambda}^{\bm p+e_j} U_{i_ji_1}^{(p_{1})}U_{i_ji_2}^{(p_{2})} \cdots U_{i_ji_{j-1}}^{(p_{{j-1}})}\xi)\\
	&=z_{j}({\bm z}^{\bm p} U_{i_ji_1}^{(p_{1})}U_{i_ji_2}^{(p_{2})} \cdots U_{i_ji_{j-1}}^{(p_{{j-1}})}\xi)\\
	&=M_{z_j} (\Delta_{1}[U_{i_ji_1}] \cdots \Delta_{j-1}[U_{i_ji_{j-1}}])({\bm z}^{\bf p} \xi).
\end{align*}
Therefore,
\[
\tau_{\Lambda}V_{i_j} \tau_\Lambda^*=M_{z_j} (\Delta_{1}[U_{i_ji_1}] \cdots \Delta_{j-1}[U_{i_ji_{j-1}}]) \quad \text{for}~~ 2 \leq j \leq l.
\]

For the remaining part, suppose $i_k \in I_n \backslash \Lambda$ for $k \in \{l+1, \dots, n\}$. Again using Proposition \ref{Remark-for unitary}, we have 
\begin{align*}
(\tau_{\Lambda}V_{i_k} \tau_\Lambda^*)({\bm z}^{\bm p} \xi)
&= \tau_{\Lambda}(V_{i_k}V_{\Lambda}^{\bm p} \xi)\\
&= \tau_{\Lambda}(V_{\Lambda}^{\bm p} V_{i_k} U_{i_k}\xi) \\
&=  \tau_{\Lambda}\left(V_{\Lambda}^{\bm p} (V_{i_k} U_{i_ki_1}^{(p_{1})} \cdots U_{i_ki_{l}}^{(p_{{l}})} \xi)\right) \\
&={\bm z}^{\bm p} (V_{i_k}|_{\cls_{\Lambda}} U_{i_ki_1}^{(p_{1})} \cdots U_{i_ki_{l}}^{(p_{{l}})}\xi) \\
&=(I_{H^2(\mathbb{D}^l)} \otimes V_{i_k}|_{\cls_{\Lambda}}) (\Delta_{1}[U_{i_ki_1}] \cdots \Delta_{l}[U_{i_ki_{l}}])({\bm z}^{\bm p} \xi).
\end{align*}
Thus, for $l+1 \leq k \leq n$, we obtain
\[
\tau_{\Lambda}V_{i_k} \tau_\Lambda^*=(I_{H^2(\mathbb{D}^l)} \otimes V_{i_k}|_{\cls_{\Lambda}})(\Delta_{1}[U_{i_ki_1}] \cdots \Delta_{l}[U_{i_ki_{l}}]).
\]

Consider now the $n$-tuple of operators $M_{\Lambda}=(M_{\Lambda,1}, \dots, M_{\Lambda,n})$ on $H^2_{\cls_{\Lambda}}(\mathbb{D}^l)$ consisting of $l$ operators 
$\{\tau_{\Lambda}V_{i_j} \tau_\Lambda^*\}_{j=1}^{l}$ and $(n-l)$ operators 
$\{\tau_{\Lambda}V_{i_k} \tau_\Lambda^*\}_{k=l+1}^{n}$, where
\begin{align}\label{Analytical Representation}
M_{\Lambda,s}	  
&=\begin{cases} 
		M_{z_1}  &\text{if $s=i_1$} \\
		M_{z_j} (\Delta_{1}[U_{i_ji_1}] \cdots \Delta_{j-1}[U_{i_ji_{j-1}}])& 
		\text{if $s=i_j$ and  $2 \leq j \leq l$}\\
		(I_{H^2(\mathbb{D}^l)} \otimes V_{i_k}|_{\cls_{\Lambda}})(\Delta_{1}[U_{i_ki_1}] \cdots \Delta_{l}[U_{i_ki_{l}}]) & \text{ if $s =i_k$ and  $l+1 \leq k \leq n $},
\end{cases} 
\end{align}
and $s \in \{1,2, \dots, n\}$. Note that
\[
\tau_{\Lambda} \cln(V_{i_j}^*)= \tau_{\Lambda} (\clh_{\Lambda} \ominus V_{i_j} \clh_{\Lambda})= H^2_{\cls_{\Lambda}}(\mathbb{D}^l) \ominus M_{\Lambda,i_j}H^2_{\cls_{\Lambda}}(\mathbb{D}^l)=\cln(M_{\Lambda,i_j}^*) 
\]
for $1\leq j \leq l$. Therefore, $\tau_{\Lambda}\clw_{\Lambda}= \bigcap\limits_{i_j \in \Lambda} \cln(M_{\Lambda, i_j}^*)$. Again for each $\bm q \in \mathbb{Z}_+^{n-l}$,
\[
\tau_{\Lambda} V_{I_n \backslash \Lambda}^{\bm q} \clw_{\Lambda}= (\tau_{\Lambda} V_{I_n \backslash \Lambda}^{\bm q}\tau_{\Lambda}^*) (\tau_{\Lambda}\clw_{\Lambda})=M_{I_n \backslash \Lambda}^{\bm q}\big(\bigcap\limits_{i_j \in \Lambda} \cln(M_{\Lambda, i_j}^*) \big).
\]
It follows that 
\begin{align}\label{Equation-Unitary equivalent on A-wandering space}
\tau_{\Lambda}(\cls_{\Lambda}(V))=  
\tau_{\Lambda}\big(\bigcap\limits_{\bm q \in \mathbb{Z}_+^{n -  l}}  V_{I_n \backslash \Lambda}^{\bm q} \clw_{\Lambda} \big)
=\bigcap\limits_{\bm q \in \mathbb{Z}_+^{n - l }}M_{I_n \backslash \Lambda}^{\bm q}\big(\bigcap\limits_{i_j \in \Lambda} \cln(M_{\Lambda, i_j}^*)\big)
=\cls_{\Lambda}(M_{\Lambda}).
\end{align}
From equation \eqref{Define unitary operator Pi }, we get $\tau_{\Lambda}(\cls_{\Lambda}(V))=\cls_{\Lambda}(V)$. Therefore, $\cls_{\Lambda}(V)=\cls_{\Lambda}(M_{\Lambda})$. Now, for each $i_k \in I_n \backslash \Lambda$ and $\xi \in \cls_{\Lambda}$, we have 
\[
M_{\Lambda,i_k} (\xi)=(\tau_{\Lambda} V_{i_k} \tau_{\Lambda}^*) (\xi)= \tau_{\Lambda} (V_{i_k} \xi) = V_{i_k}\xi
\]
as $\tau_{\Lambda}^*\xi =\xi$ and $V_{i_k} \xi \in \cls_{\Lambda}$. Hence 
$ M_{\Lambda,i_k}|_{\cls_{\Lambda}}= V_{i_k}|_{\cls_{\Lambda}}$.

We summarize the previous deliberation as follows:

\begin{prop}\label{Proposition-Analytic model}
Let $\Lambda \subseteq I_n$, and $\clh=\bigoplus\limits_{\Lambda \subseteq I_n} \clh_{\Lambda}$ be the {\textit{Wold decomposition}} for an $n$-tuple of isometries $V=(V_1,\ldots, V_n)$ with equal range on $\clh$. If $\cls_{\Lambda}(V)=\bigcap\limits_{\bm q \in \mathbb{Z}_+^{n -  |\Lambda|}}  V_{I_n \backslash \Lambda}^{\bm q} \clw_{\Lambda}$ is a non-zero subspace for $\clh_{\Lambda}$, then $V|_{\clh_\Lambda}=(V_1|_{\clh_{\Lambda}}, \dots, V_n|_{\clh_{\Lambda}})$ on $\clh_\Lambda$ is unitarily equivalent to $M_{\Lambda}= (M_{\Lambda,1},  \ldots, M_{{\Lambda},n}) $ on $H^2_{\cls_{\Lambda}}(\mathbb{D}^{|\Lambda|})$, where $M_{\Lambda, s} $'s are defined by equation \eqref{Analytical Representation}. If $I_n \backslash \Lambda= \{i_{l+1}< \dots <i_n\}$, then the wandering data for $M_{\Lambda}$ is expressed as 
\[
wd_{\Lambda}(M_{\Lambda})= (I|_{\cls_{\Lambda}}, V_{i_{l+1}}|_{\cls_{\Lambda}},\dots, V_{i_n}|_{\cls_{\Lambda}})
\] 
and the other wandering data for $M_{\Lambda}$ are zero tuples.
\end{prop}

We refer $M_{\Lambda}$ as the model operator corresponding to each $\Lambda  \subseteq I_n$. It is easy to see that the model operator $M_{\Lambda}= (M_{\Lambda,1},  \ldots, M_{{\Lambda},n})$ is an $n$-tuple of {\textit{isometries with equal range}} on $H^2_{\cls_{\Lambda}}(\mathbb{D}^{|\Lambda|})$. Our goal is now to formulate an analytic model for an $n$-tuple of isometries $V=(V_1,\ldots, V_n)$ with equal range on $\clh$. To do that, we will apply Proposition \ref{Proposition-Analytic model} to each $\clh_{\Lambda}$ for $\Lambda (\neq \emptyset) \subseteq I_n$ and then assemble all the pieces together. Indeed, we have the following result:

\begin{thm}
Let $(V_1, \ldots, V_n)$ be an $n$-tuple of {\textit{isometries with equal range}} on $\clh$. Then $(V_1,\ldots, V_n)$ is unitarily equivalent to $(M(V_1),\ldots, M(V_n))$ on $\bigoplus\limits_{\Lambda \subseteq I_n}H^2_{\cls_{\Lambda}}(\mathbb{D}^{|\Lambda|})$.
\end{thm}

\begin{proof}
Suppose $V=(V_1, \ldots, V_n)$ is an $n$-tuple of {\textit{isometries with equal range}} on $\clh$. Then Theorem \ref{Theorem Main Result} yields that there exist $2^n$ joint $V$-reducing orthogonal decomposition spaces $\clh_{\Lambda}$ such that $\clh= \bigoplus\limits_{\Lambda \subseteq I_n}\clh_{\Lambda}$. Moreover, 
\[
V_j= \bigoplus\limits_{\Lambda \subseteq I_n}V_j|_{\clh_{\Lambda}} \in \clb(\bigoplus\limits_{\Lambda \subseteq I_n}\clh_{\Lambda})
\]
for each $j \in I_n$. Now Proposition \ref{Proposition-Analytic model} asserts 
that for each $\Lambda (\neq \emptyset) \subseteq I_n$, $V|_{\clh_\Lambda}=(V_1|_{\clh_{\Lambda}}, \dots, V_n|_{\clh_{\Lambda}})$	 on $\clh_\Lambda$ is unitarily equivalent to $M_{\Lambda} =(M_{\Lambda,1}, \dots, M_{\Lambda,n})$ on $H^2_{\cls_{\Lambda}}(\mathbb{D}^{|\Lambda|})$ via the unitary map $\tau_{\Lambda}:\clh_\Lambda \rightarrow H^2_{\cls_{\Lambda}}(\mathbb{D}^{|\Lambda|})$, where $\cls_{\Lambda}$ is the $\Lambda$-wandering subspace for $V$.
Consider 
\[
M(V_j)=\bigoplus\limits_{\Lambda \subseteq I_n} M_{\Lambda,j} \in \clb\big(\bigoplus\limits_{\Lambda \subseteq I_n}H^2_{\cls_{\Lambda}}(\mathbb{D}^{|\Lambda|})\big)
\]	
for each $j \in I_n$. Define $\tau=\bigoplus\limits_{\Lambda \subseteq I_n} \tau_{\Lambda}$. Clearly, $\tau$ is a unitary operator from $\clh $ to 
$ \bigoplus\limits_{\Lambda \subseteq I_n}H^2_{\cls_{\Lambda}}(\mathbb{D}^{|\Lambda|})$. Therefore, for each $j \in I_n$, we have
\[
\tau V_j\tau^*=\big(\bigoplus\limits_{\Lambda \subseteq I_n} \tau_{\Lambda}\big)\big(\bigoplus\limits_{\Lambda \subseteq I_n}V_j|_{\clh_{\Lambda}} \big)\big(\bigoplus\limits_{\Lambda \subseteq I_n} \tau_{\Lambda}^*\big)
=\bigoplus\limits_{\Lambda \subseteq I_n} \tau_{\Lambda} {V_j}|_{\clh_{\Lambda}}\tau_{\Lambda}^*=\bigoplus_{\Lambda \subseteq I_n}M_{\Lambda,j}=M(V_j). 
\] 
Hence $(V_1,\dots, V_n)$ is unitarily equivalent to $(M(V_1),\dots, M(V_n))$ on $\bigoplus\limits_{\Lambda \subseteq I_n}H^2_{\cls_{\Lambda}}(\mathbb{D}^{|\Lambda|})$.  
\end{proof}

\section{Invariants}

In this section we show that the wandering data are completely unitary invariants for
$n$-tuples of {\textit{isometries with equal range}}.

\begin{thm}
Let $V=(V_1,\ldots,V_n)$ on $\clh$ and $T=(T_1, \ldots, T_n)$ on $\clk$ be two $n$-tuples of {\textit{isometries with equal range}}. Then the following are equivalent:
\begin{enumerate}
\item $(V_1, \ldots,V_n)$ is unitarily  equivalent to $ (T_1,\ldots, T_n)$. 
\item For all $\Lambda \subseteq I_n$, $(I|_{\cls_{\Lambda}(V)},V_{i_{l+1}}|_{\cls_{\Lambda}(V)},\ldots, V_{i_n}|_{\cls_{\Lambda}(V)})$ is unitarily  equivalent 
to \\ $(I|_{\cls_{\Lambda}(T)},
T_{i_{l+1}}|_{\cls_{\Lambda}(T)},\ldots, T_{i_n}|_{\cls_{\Lambda}(T)})$, 
where $i_k \in I_n \backslash \Lambda$ for $ k \in \{l+1,\dots,n\}$. 
\end{enumerate}
\end{thm}

\begin{proof}
$(1) \implies (2):$
Suppose that the tuples $V$ and $T$ are unitarily equivalent. Then there exists a unitary map $\Pi:\clh \rightarrow \clk$ such that $\Pi V_i  =T_i\Pi$ for $1 \leq i \leq n$. Now for each  $\Lambda=\{i_1<\dots < i_l \} \subseteq I_n$, we have 
\[
\Pi(\clw_{\Lambda}(V))=\bigcap\limits_{i_j \in \Lambda} \Pi\cln (V_{i_j}^*)= \bigcap\limits_{i_j \in \Lambda} \cln (T_{i_j}^*)= \clw_{\Lambda}(T)
\]
as $\Pi\cln (V_{i_j}^*) =\Pi\clh \ominus \Pi V_{i_j} \clh= \clk \ominus T_{i_j}\clk=\cln (T_{i_j}^*)$ for $ 1 \leq j \leq l$. Therefore,
\[
\Pi\cls_{\Lambda}(V)= \bigcap\limits_{\bm q \in \mathbb{Z}_+^{n-|\Lambda|}}
(\Pi V_{I_n \backslash \Lambda}^{\bm q} \Pi^*) \clw_{\Lambda}(T)
=\bigcap\limits_{\bm q \in \mathbb{Z}_+^{n-|\Lambda|}}
T_{I_n \backslash \Lambda}^{\bm q}  \clw_{\Lambda}(T)= \cls_{\Lambda}(T)
\] 
where $\cls_{\Lambda}(V)$ and $\cls_{\Lambda}(T)$ are $\Lambda$-wandering subspaces for the tuples $V$ and $T$, respectively. Thus $\Pi|_{\cls_{\Lambda}(V)}:\cls_{\Lambda}(V) \rightarrow \cls_{\Lambda}(T)$ is unitary. Again Lemma  \ref{Lemma- Reduce A wandering subspace } says that $\cls_{\Lambda}(V)$ reduces $V_{i_k}$ for $i_k \in I_n \backslash \Lambda= \{i_{1+1} <\dots < i_n \}$. Then,  for $i_k \in I_n \backslash \Lambda $, $\bm q \in \mathbb{Z}^{n- |\Lambda|}$ and $\eta \in \clw_{\Lambda}$, we obtain
\begin{align*}
\big(\Pi|_{\cls_{\Lambda}(V)} V_{i_k}|_{\cls_{\Lambda}(V)}\big) V_{I_n \backslash \Lambda}^{\bm q} \eta
= (\Pi V_{i_k}) V_{I_n \backslash \Lambda}^{\bm q} \eta
=(T_{i_k} \Pi)V_{I_n \backslash \Lambda}^{\bm q} \eta 
&=\big(T_{i_k}|_{\cls_{\Lambda}(T)}\Pi|_{\cls_{\Lambda}(V)}\big) V_{I_n \backslash \Lambda}^{\bm q} \eta 
\end{align*}
as $V_{I_n \backslash \Lambda}^{\bm q} \eta \in \cls_{\Lambda}(V)$. Hence there 
exists a unitary operator
$\Pi|_{\cls_{\Lambda}(V)}:\cls_{\Lambda}(V) \rightarrow \cls_{\Lambda}(T)$ such that  $\Pi|_{\cls_{\Lambda}(V)} V_{i_k}|_{\cls_{\Lambda}(V)}=T_{i_k}|_{\cls_{\Lambda}(T)}  \Pi|_{\cls_{\Lambda}(V)}$ for ${i_k} \in I_n \backslash \Lambda$. Thus $(I|_{\cls_{\Lambda}(V)},V_{i_{l+1}}|_{\cls_{\Lambda}(V)},\ldots, V_{i_n}|_{\cls_{\Lambda}(V)})$  is unitarily  equivalent to  $(I|_{\cls_{\Lambda}(T)},T_{i_{l+1}}|_{\cls_{\Lambda}(T)},\ldots, T_{i_n}|_{\cls_{\Lambda}(T)})$ for all $\Lambda \subseteq I_n$.

$(2) \implies (1):$
Assume that for each $\Lambda =\{i_1<\dots< i_l\}  \subseteq I_n$, there exists a unitary map $\pi_{\Lambda}: \cls_{\Lambda}(V) \rightarrow \cls_{\Lambda}(T)$ such that $\pi_{\Lambda}V_{i_k}|_{\cls_{\Lambda}(V)}=T_{i_k}|_{\cls_{\Lambda}(T)}\pi_{\Lambda}$ where $i_k \in I_n \backslash \Lambda$ for $ k \in \{l+1, \dots,n\}$. Now Theorem \ref{Theorem Main Result} asserts that $\clh=\bigoplus\limits_{\Lambda \subseteq I_n} \clh_{\Lambda}$ and $ \clk= \bigoplus\limits_{\Lambda \subseteq I_n} \clk_{\Lambda}$. Moreover, $\clh_{\Lambda}= \bigoplus\limits_{\bm p \in \mathbb{Z}_+^{l}} V_{\Lambda}^{\bm p}\cls_{\Lambda}(V)$ and $\clk_{\Lambda}= \bigoplus\limits_{\bm p \in \mathbb{Z}_+^{l}} V_{\Lambda}^{\bm p}\cls_{\Lambda}(T)$.  
	
For each $\Lambda  \subseteq I_n$, define a unitary operator $\Phi_{\Lambda}: \clh_{\Lambda} \rightarrow \clk_{\Lambda}$ by
\[
\Phi_{\Lambda}( V_{\Lambda}^{\bm p} U_{i_j}\xi) :=  T_{\Lambda}^{\bm p}(\widetilde{U}_{i_j} \pi_{\Lambda}\xi) \quad ( \bm p =(p_1, \ldots, p_l) \in \mathbb{Z}_+^l,\  \ \xi \in \cls_{\Lambda}(V))
\]
where the unitaries $U_{i_j} \in \clb(\clh), \widetilde{U}_{i_j} \in \clb(\clk)$ are defined as (see Proposition \ref{Remark-for unitary}, and Remark \ref{Remark-Reducingspace-unitary}) 
\begin{align}\label{Equation for Unitary $U_{i_j}$} 
U_{i_j}=
\begin{cases}
V_{\Lambda}^{*(\bm p +e_j)} V_{i_j}V_{\Lambda}^{\bm p}  &\text{if} \quad i_j \in \Lambda, \\
V_{\Lambda}^{*\bm p} V_{i_j}V_{\Lambda}^{\bm p}  &\text{if} \quad i_j \in I_n \backslash \Lambda, 
\end{cases}	 
\end{align}
and 
\begin{align} \label{Equation for Unitary $U_{i_k}$} 
\widetilde{U}_{i_j}=
\begin{cases}
T_{\Lambda}^{*(\bm p +e_j)} T_{i_j}T_{\Lambda}^{\bm p}  &\text{if} \quad i_j \in \Lambda, \\
T_{\Lambda}^{*\bm p} T_{i_j}T_{\Lambda}^{\bm p}  &\text{if} \quad i_j \in I_n \backslash \Lambda.
\end{cases}	 
\end{align}
In particular, if $i_j=i_1 \in \Lambda $, then $U_{i_j}=I_{\clh} $ and 
$\widetilde{U}_{i_j}=I_{\clk}$. Thus $\Phi_{\Lambda}(V_{\Lambda}^{\bm{p}} \xi)=T_{\Lambda}^{\bm p}(\pi_{\Lambda} \xi)$ for $\bm p \in \mathbb{Z}_+^l, \xi \in \cls_{\Lambda}(V)$.	 
Now for each $i_j \in \Lambda$, using equation \eqref{Equation for Unitary $U_{i_j}$}, 
we get
\begin{align*}
(\Phi_{\Lambda} V_{i_j}|_{\clh_{\Lambda}} )(V_{\Lambda}^{\bm{p}} \xi)
		&=\Phi_{\Lambda}( V_{i_j}V_{\Lambda}^{\bm{p}} \xi) \\
		&=\Phi_{\Lambda}( V_{\Lambda}^{\bm p+e_j} U_{i_j}\xi) \\
		&=T_{\Lambda}^{\bm p+e_j}(\widetilde{U}_{i_j}\pi_{\Lambda} \xi) \\
		&=T_{i_j}|_{\clk_{\Lambda}}T_{\Lambda}^{\bm p}(\pi_{\Lambda}\xi)\\
		&=T_{i_j}|_{\clk_{\Lambda}}\Phi_{\Lambda}( V_{\Lambda}^{\bm p} \xi).
\end{align*}
It follows that 
\[
\Phi_{\Lambda} V_{i_j}|_{\clh_{\Lambda}}\Phi_{\Lambda}^*= T_{i_j}|_{\clk_{\Lambda}}\quad  \text{for} \quad i_j \in \Lambda. 
\]
For $i_k \in I_n \backslash \Lambda$, using equation \eqref{Equation for Unitary $U_{i_k}$}, we obtain
\begin{align*}
(\Phi_{\Lambda} V_{i_k}|_{\clh_{\Lambda}} )(V_{\Lambda}^{\bm{p}} \xi)
&=\Phi_{\Lambda}(V_{\Lambda}^{\bm{p}}  U_{i_k}\xi)\\
&= T_{\Lambda}^{\bm p}(\widetilde{U}_{i_k}\pi_{\Lambda}  \xi)\\
&=T_{i_k}T_{\Lambda}^{\bm p} (\pi_{\Lambda}\xi)\\
&=T_{i_k}|_{\clk_{\Lambda}}\Phi_{\Lambda}(V_{\Lambda}^{\bm p} \xi).
\end{align*}
Therefore,
\[
\Phi_{\Lambda} V_{i_k}|_{\clh_{\Lambda}}\Phi_{\Lambda}^*= T_{i_k}|_{\clk_{\Lambda}} \quad \text{for} \quad i_k \in I_n \backslash \Lambda.
\]
Again Theorem \ref{Theorem Main Result} yields that $V_j=\bigoplus\limits_{\Lambda \subseteq I_n} V_j|_{\clh_{\Lambda}}$ and $T_j=\bigoplus\limits_{\Lambda \subseteq I_n} T_j|_{\clk_{\Lambda}}$ for each $j=1,2,\dots,n$. We now consider the unitary operator $\Phi=\bigoplus\limits_{\Lambda \subseteq I_n} \Phi_{\Lambda}: \clh \rightarrow \clk$. Then for each $j \in\{1,2,\dots,n\}$, 
\[
\Phi V_j \Phi^*=\big(\bigoplus\limits_{\Lambda \subseteq  I_n} \Phi_{\Lambda}\big)
\big(\bigoplus\limits_{\Lambda \subseteq I_n} V_j|_{\clh_{\Lambda}}\big)
\big(\bigoplus\limits_{\Lambda \in I_n} \Phi_{\Lambda}^*\big)
=\bigoplus\limits_{\Lambda \subseteq I_n} \Phi_{\Lambda} V_j|_{\clh_{\Lambda}} \Phi_{\Lambda}^*=\bigoplus\limits_{\Lambda \subseteq I_n}  T_j|_{\clk_{\Lambda}} 
=T_j.
\]
This completes the proof.
\end{proof}


\NI\textit{Acknowledgements:} 
The anonymous reviewer's critical and helpful comments and suggestions have greatly enhanced the paper's presentation, for which the authors are grateful.
The second author's research work is partially supported by a Faculty Initiation Grant (FIG scheme), IIT Roorkee (Ref. No: MAT/FIG/100820) and
the Mathematical Research Impact Centric Support (MATRICS) (MTR/2021/000695), 
SERB (DST), Government of India.
\vspace{0.2cm}

\NI\textit{Data availability:}
Data sharing is not applicable to this article as no data sets
were generated or analysed during the current study.
\vspace{0.2cm}

\NI\textit{Declarations}
\vspace{0.2cm}

\NI\textit{Conflict of interest:}
The authors have no competing interests to declare.

\end{document}